\documentclass[a4paper, reqno]{amsart}
\usepackage[utf8]{inputenc}
\usepackage[T1]{fontenc}
\usepackage[english]{babel}
\usepackage{amsmath}
\usepackage{amsthm}
\usepackage{xfrac}
\usepackage{amssymb, enumerate}
\usepackage{amsfonts}
\usepackage{dsfont}
\usepackage{graphicx}
\usepackage[lofdepth,lotdepth]{subfig}
\usepackage{mathtools,mathrsfs} %nice calligraphy
\usepackage{cases} %for accolades with numbered equations
\usepackage{tikz}
\usepackage[unicode,hypertexnames=false,colorlinks=true,linkcolor=blue,citecolor=blue]{hyperref}
\usepackage{xfrac}

\usepackage[shortlabels]{enumitem}

\usepackage{hyperref} %Leave it as the but one  last package
\usepackage[capitalise, nameinlink, noabbrev]{cleveref} %Leave it as the last package
\hypersetup{colorlinks=true,
	linkcolor=blue,
	citecolor=red,
}

\usepackage{tikz}
\newcommand{\N}{\mathbb{N}}

\newcommand{\R}{\mathbb{R}}
\newcommand{\eps}{\varepsilon}

\newcommand{\Om}{\Omega}

\newcommand{\ds}{\displaystyle}
\DeclareMathOperator{\cp}{cap}
\newcommand{\mean}[1]{\,-\hskip-1.08em\int_{#1}} %media integrale displayed
\def\HH{\mathcal{H}}

\def\XXint#1#2#3{{\setbox0=\hbox{$#1{#2#3}{\int}$ }
\vcenter{\hbox{$#2#3$ }}\kern-.6\wd0}}

\newcommand{\mathbbmm}[1]{\text{\usefont{U}{bbm}{m}{n}#1}}
\newcommand{\ind}{\mathbbmm{1}}

\newtheorem{proposition}{Proposition}[section]
\newtheorem{theorem}[proposition]{Theorem}
\newtheorem{corollary}[proposition]{Corollary}
\newtheorem{lemma}[proposition]{Lemma}

\theoremstyle{definition}
\newtheorem{definition}[proposition]{Definition}
\newtheorem{remark}[proposition]{Remark}
\newcommand{\beq}{\begin{equation}}
\newcommand{\eeq}{\end{equation}}
\newcommand{\ben}{\begin{enumerate}}
\newcommand{\een}{\end{enumerate}}
\newcommand{\bit}{\begin{itemize}}
\newcommand{\eit}{\end{itemize}}

\title{Regularity of the optimal sets for the second Dirichlet eigenvalue}

\author{Dario Mazzoleni, Baptiste Trey, Bozhidar Velichkov}

\address{Dario Mazzoleni: \newline \indent
Dipartimento di Matematica ``F. Casorati'',
Universit\`a di Pavia,\newline \indent
Via Ferrata 5, I--27100 Pavia, Italy,} 
\email{dario.mazzoleni@unipv.it}

\address{Baptiste Trey: \newline \indent
Universit\'e Grenoble Alpes, CNRS UMR 5582, Institut Fourier \newline \indent
100 rue des Math\'ematiques, F--38610 Gi\`eres, France
} 
\email{baptiste.trey@etu.univ-grenoble-alpes.fr}

\address {Bozhidar Velichkov: \newline \indent
Dipartimento di Matematica, Universit\`a di Pisa \newline \indent
Largo Bruno Pontecorvo, 5, 56127 Pisa - ITALY}
\email{bozhidar.velichkov@unipi.it}

\thanks{{\bf Acknowledgments.} 
D.M. has been partially supported by the INdAM-GNAMPA project 2019 ``Ottimizzazione spettrale non lineare''. B.V.  has been partially supported by the European Research Council (ERC), under the European Union’s Horizon 2020 research and innovation programme, through the project ERC VAREG - \it Variational approach to the regularity of the free boundaries \rm (grant agreement No. 853404).}


\begin{document}

\begin{abstract}
This paper is dedicated to the regularity of the optimal sets for the second eigenvalue of the Dirichlet Laplacian. Precisely, we prove that if the set $\Omega$ minimizes the functional 
$$\mathcal F_\Lambda(\Omega)=\lambda_2(\Omega)+\Lambda |\Omega|,$$
among all subsets of a smooth bounded open set $D\subset \R^d$, where $\lambda_2(\Omega)$ is the second eigenvalue of the Dirichlet Laplacian on $\Omega$ and $\Lambda>0$ is a fixed constant, then $\Omega$ is equivalent to the union of two disjoint open sets $\Omega_+$ and $\Omega_-$, which are $C^{1,\alpha}$-regular up to a (possibly empty) closed set of Hausdorff dimension at most $d-5$, contained in the one-phase free boundaries $D\cap \partial\Omega_+\setminus\partial\Omega_-$ and $D\cap\partial\Omega_-\setminus\partial\Omega_+$. 
\end{abstract}
\maketitle

\section{Introduction}\label{s:intro}
%\subsection{Foreword}
%In the last few years, there has been a lot of interest in the study of spectral optimization problems, in particular those involving eigenvalues of the Dirichlet Laplacian.
%The motivations for this interest are both from the theoretical point of view and from the applications.
%The problem of existence of optimal shapes is now rather well understood, while their regularity is related to the so called \emph{Bernoulli} free boundary problems and many open problems concerning this topic are still open.

%This paper is dedicated to the regularity of the optimal sets for the second eigenvalue of the Dirichlet Laplacian. 
Given a real constant $\Lambda>0$ and an open set $\Omega\subset\R^d$, we define 
\begin{equation}\label{e:fun}
\mathcal F_\Lambda(\Omega)=\lambda_2(\Omega)+\Lambda|\Omega|\ ,
\end{equation}
where $|\Omega|$ is the Lebesgue measure of the set $\Omega$ and $\lambda_2(\Omega)$ is the second eigenvalue (counted with the due multiplicity) of the Laplace operator in $\Omega$, with Dirichlet boundary conditions on $\partial\Omega$. Precisely, we recall the following variational characterization of the second eigenvalue:
\begin{equation}\label{e:lambda_2}
\lambda_2(\Omega)=\min_{E_2\subset H^1_0(\Omega)}\max\Big\{\int_{\Omega}|\nabla u|^2\,dx\ :\ u\in E_2,\ \int_{\Omega}u^2\,dx=1\Big\},
\end{equation}
where the minimum is taken among all two-dimensional subspaces $E_2$ of the Sobolev space $H^1_0(\Omega)$, which is the closure, with respect to the $H^1$ norm, 
of the space $C^{\infty}_c(\Omega)$ of smooth functions compactly supported in $\Omega$.
\medskip

This paper is dedicated to the regularity of the sets that minimize $\mathcal F_\Lambda$. Our main result is the following.

%In this paper we prove a regularity result for the open sets $\Omega$ that minimize $\mathcal F_\Lambda$ among all open subsets $\Omega$ of a smooth bounded open set $D\subset\R^d$.

%In particular, in dimension $d\le 4$, we show that any solution is equivalent to the union of two disjoint  $C^{1,\alpha}$-regular open sets contained in $D$. 
\begin{theorem}\label{thm:main}
	Let $D\subset \R^d$ be an open  bounded set of class $C^{1,\beta}$, for some $\beta>0$, and let $\Lambda>0$ be a given constant. 
Let $\Omega\subset D$ be an open set that minimizes $\mathcal F_\Lambda$ in $D$, that is, 
\begin{equation}\label{e:min-pb-open}
\mathcal F_\Lambda(\Omega)\le \mathcal F_\Lambda(\widetilde\Omega)\quad\text{for every open set}\quad \widetilde\Omega\subset D.
\end{equation}
Then, there are two disjoint open sets $\Omega_+$ and $\Omega_-$, both contained in $\Omega$, such that 
$$\lambda_2(\Omega_+\cup\Omega_-)=\lambda_2(\Omega)\qquad \text{and}\qquad |\,\Omega\,\setminus\,(\Omega_+\cup\Omega_-)|=0.$$
Each of the boundaries $\partial \Omega_+$ and $\partial\Omega_-$ can be decomposed as the disjoint union of a regular and of a (possibly empty) singular part, namely
$$\partial \Omega_\pm=\text{\rm Reg}\,(\partial \Omega_\pm)\cup \text{\rm Sing}\,(\partial \Omega_\pm),$$
with the following properties:
	\begin{enumerate}[\quad\rm(i)]
		\item The regular set $\text{\rm Reg}(\partial \Omega_\pm)$ is an open subset of $\partial \Omega_\pm$, which is locally the graph of a $C^{1,\alpha}$ function, for some $\alpha>0$. Moreover, $\text{\rm Reg}(\partial \Omega_\pm)$ contains both the two-phase free boundary $\partial\Omega_+\cap\partial\Omega_-$ and the contact sets with the boundary of the box: $\partial\Omega_+\cap\partial D$ and $\ \partial\Omega_-\cap\partial D$.\medskip
		\item The singular set $\text{\rm Sing}(\partial \Omega_\pm)$ is a closed subset of $\partial \Omega_\pm$ and contains only one-phase points. Moreover, there exists a critical dimension $d^*\in\{5,6,7\}$ (see \cref{oss:crit}) such that 
		\begin{itemize}
			\item if $d<d^*$, then the singular set is empty,
			\item if $d=d^*$, then the singular set consists of a finite number of points,
			\item if $d>d^*$, then the singular set has Hausdorff dimension at most $d-d^*$.
		\end{itemize} 
	\end{enumerate}
\end{theorem}

\begin{remark}\label{oss:crit}
The critical dimension $d^\ast$ is the lowest dimension in which there exist minimizing one-phase free boundaries with singularities. It is known that $d^*$ is $5$, $6$ or $7$ (see~\cite{jesa} and the references therein).
\end{remark}	
\begin{remark}
The existence of an open set $\Omega$ that minimizes $\mathcal F_\Lambda$ in $D$ was already proved in \cite[Corollary~5.11]{bucve}.
\end{remark}

\subsection{Optimal sets for the eigenvalues of the Dirichlet Laplacian: an overview}\label{sub:history} 
The optimization problems for functionals involving the eigenvalues of an elliptic operator and the volume (the Lebesgue measure) allow to achieve a better understanding on the interaction between the geometry (the shape) of the domains in $\R^d$ and their spectrum. In the particular case when the elliptic operator is the Laplacian with Dirichlet boundary conditions, these variational problems have a rich, century-long history. We will briefly recall the main results concerning the existence and the regularity of optimal sets and we will refer to the survey papers \cite{bu} and \cite{henrot} for a more detailed introduction to the topic.

\subsubsection{Functionals involving only the first eigenvalue}
%The existence and the regularity of optimal sets for functionals involving the spectrum of the Dirichlet Laplacian and the Lebesgue measure have been widely studied for more than a century.
For the principal eigenvalue $\lambda_1$, the analogous of \eqref{e:fun} is the variational problem
\begin{equation}\label{eq:main007}
\min\big\{\lambda_1(A)+\Lambda|A| : A\subset D\big\}.
\end{equation}
We first recall that the classical Faber-Krahn inequality implies that if $\Lambda$ is big enough, then balls are the only (up to translation in $D$) solutions. 
On the other hand, if $\Lambda$ is small, then the existence of a minimizer of \eqref{eq:main007} in the class of quasi-open sets can be easily proved (see \cref{s:existenceqo}), but the optimal shapes $\Omega\subset D$ are in general not explicit. In this case, the regularity of the free boundary $\partial\Omega\cap D$ (the part contained in the box $D$) was obtained by Brian\c{c}on and Lamboley in~\cite{brla}.  In fact, by the variational characterization of $\lambda_1(\Omega)$, the problem \eqref{eq:main007} is equivalent to the following variational problem involving functions and not sets
$$\min\left\{\int_D|\nabla u|^2\,dx+\Lambda|\{u>0\}|\ :\ u\in H^1_0(D),\ \int_Du^2\,dx=1\right\}.$$
%Now, %since both the Dirichlet Energy and the Lebesgue measure
%both the Dirichlet Energy $\ds\int_D|\nabla u|^2\,dx$ and the Lebesgue measure $|\{u>0\}|$ 
If $u$ is a minimizer and $\xi\in C^\infty_c(D;\R^d)$ is a smooth vector field, then the function
%are differentiable along smooth vector fields, that is, for every $\xi\in C^\infty_c(D;\R^d)$, we have that the function 
$$t\mapsto \int_D|\nabla u_t|^2\,dx+\Lambda|\{u_t>0\}|\qquad\text{where}\qquad u_t(x):=u(x+t\xi(x)),$$ 
is differentiable and has minimum at $t=0$. The associated first order optimality condition gives that, in some suitable sense, $u$ is a solution of the following (one-phase) free boundary problem 
\begin{equation}\label{e:over_bvp}
-\Delta u=\lambda_1(\Omega)u\quad\text{in}\quad\Omega,\qquad |\nabla u|=\sqrt{\Lambda}\quad\text{on}\quad \partial\Omega\cap D,\qquad \Omega=\{u>0\}\subset D,
\end{equation}
for which one can apply the techniques developed by Alt and Caffarelli in~\cite{alca} for the one-phase Bernoulli problem
\begin{equation*}
-\Delta u=0\quad\text{in}\quad\{u>0\},\qquad |\nabla u|=\sqrt{\Lambda}\quad\text{on}\quad \partial\{u>0\}\cap D,
\end{equation*}
obtained from the minimization (with suitable Dirichlet boundary conditions on $\partial D$) of the functional 
$$\int_D|\nabla u|^2\,dx+\Lambda|\{u>0\}\cap D|.$$
Finally, we notice that, for solutions $\Omega$ of \eqref{eq:main007}, the regularity of the full boundary $\partial\Omega$, including the part touching $\partial D$, was obtained was obtained recently in \cite[Theorem~1.2]{rtv} (by an argument relying on \cite{savin}) and in \cite{stv} (by the epiperimetric inequality from \cite{spve}) .

\subsubsection{Functionals involving higher eigenvalues} For what concerns functionals $\mathcal F$ depending on the higher eigenvalues of the Dirichlet Laplacian, the existence of minimizers is known only in the class of quasi-open sets (the definition of a quasi-open set is recalled in \cref{sub:intro-quasi-open}). In this class of domains, Buttazzo and Dal Maso \cite{budm} proved the existence of optimal sets for general shape optimization problems
\begin{equation}\label{eq:main0077}
\min\big\{\mathcal F(A)+\Lambda|A| : A\subset D,\ A\ \text{quasi-open}\big\},
\end{equation}
involving functionals $\mathcal F$ of the form 
$$\mathcal F(\Omega)=F\big(\lambda_1(\Omega),\dots,\lambda_k(\Omega)\big),$$
for which the function $F:\R^k\to\R$ satisfies only some mild semicontinuity and monotonicity assumptions. 
%which in particular covers the two model cases 
%$$\mathcal F(\Omega)=\lambda_k(\Omega)\qquad\text{and}\qquad \mathcal F(\Omega)=\sum_{j=1}^k\lambda_j(\Omega).$$

The regularity of the optimal sets in this more general situation is still to be completly understood even for the simplest model case 
$$F\big(\lambda_1,\dots,\lambda_k\big)=\lambda_k.$$ 
The main difficulty is in the fact that the higher eigenvalues $\lambda_k(\Omega)$ of the Dirichlet Laplacian are variationally characterized  by the following min-max principle
$$\lambda_k(A):=\min_{E_k\subset H^1_0(A)}\max_{u\in E_k\setminus \{0\}}\frac{\ds\int_A|\nabla u|^2\,dx}{\ds\int_Au^2\,dx}\ ,\qquad k\in\N,$$
where the minimum is taken over all $k$-dimensional linear subspaces $E_k$ of $H^1_0(A)$. One consequence of this min-max formulation is that, for $k\ge 2$, the functional $\Omega\mapsto\lambda_k(\Omega)$ is not differentiable with respect to variations of the domain $\Omega$ along smooth vector fields (see for instance \cite{henrot-book}), which in particular means that one cannot write an overdetermined boundary value problem as \eqref{e:over_bvp} for just one of the associated eigenfunctions $u_k\in E_k$.
% In fact, only few regularity results have been achieved in this context. 
\smallskip

Several results were obtained recently for functionals involving not only higher eigenvalues but also the principal one $\lambda_1(\Omega)$. In fact, for this type of {\it non-degenerate} functionals the regularity of the free boundary $\partial\Omega\cap D$ of an optimal set $\Omega$ was recently proved in \cite{krli} and \cite{mtv} (see also \cite{trey2} for the case of  more general operators). The main model example of such a functional is 
$$\ds F\big(\lambda_1,\dots,\lambda_k\big)=\sum_{j=1}^k\lambda_j,$$
and the crucial observation is that the vector-valued function $U=(u_1,\dots,u_k):D\to\R^k$, whose components are the first $k$ eigenfunctions on $\Omega$, is a solution of a free boundary problem:
\begin{equation}\label{e:over_bvp_vect}
-\Delta u_j=\lambda_j(\Omega)u_j\quad\text{in}\quad\Omega\quad\text{for}\quad j=1,\dots,k\,;\qquad \big|\nabla |U|\big|=\sqrt{\Lambda}\quad\text{on}\quad \partial\Omega\cap D,\qquad \Omega=\{|U|>0\}\subset D,
\end{equation}
which is closely related to the vectorial Bernoulli problem 
\begin{equation}
-\Delta U=0\quad\text{in}\quad\{|U|>0\},\qquad \big|\nabla |U|\big|=\sqrt{\Lambda}\quad\text{on}\quad \partial\{|U|>0\}\cap D,
\end{equation}
obtained from the minimization of the functional 
$$\int_D|\nabla U|^2\,dx+\Lambda\big|\{|U|>0\}\cap D\big|,$$
and which was studied in \cite{csy}, \cite{mtv}, \cite{spve}, \cite{desilva-tortone}, and \cite{mtv2}. 
\smallskip

For what concerns the optimal sets for {\it degenerate} functionals of the form $\mathcal F(\Omega)=\lambda_k(\Omega)$, the only available regularity result for $k\ge 2$ was obtained by Kriventsov and Lin in \cite{krli2}, where they prove both the existence of an open optimal set $\Omega$ and the $C^{1,\alpha}$-regularity of the flat part of the free boundary. The full regularity of the optimal sets is still not completely understood, as $\partial\Omega$ might contain cusp-like singularities (branching points), which a priori might  be a large set of the same dimension as the free boundary. \smallskip

The aim of the present paper is to give a complete description, including the branching points as well as the contact points with $\partial D$, of the boundary of the optimal sets for the second eigenvalue ($k=2$).

\subsubsection{Multiphase shape optimization problems} 
The variational minimization problem
\begin{equation}\label{e:sop:lambda2:intro}
\min\Big\{\lambda_2(\Omega)+\Lambda|\Omega|\ :\ \Omega\ \text{quasi-open},\ \Omega\subset D\Big\}
\end{equation} 
is related to a class of spectral optimization problems involving multiple disjoint sets, the so-called multiphase shape optimization problems. Indeed, \eqref{e:sop:lambda2:intro} is equivalent to the variational problem 
\begin{equation}\label{e:multi2}
\min\Big\{\max\big\{\lambda_1(\Omega_1);\lambda_1(\Omega_2)\big\}+\Lambda|\Omega_1\cup\Omega_2|\ :\ \ \Omega_1\ \text{ and }\ \Omega_2\ \text{ are disjoint quasi-open subsets of }\ D\Big\}.
\end{equation}
We notice that this multiphase version of \eqref{e:sop:lambda2:intro} was already exploited in \cite{bucve} in order to prove the existence of open optimal sets for the functional $\mathcal F_\Lambda=\lambda_2+\Lambda|\cdot|$ in $D$. In the present paper, we will use an equivalent free boundary version (see \cref{s:equivalencefb}).

The study of variational problems for functionals of the form 
$$\mathcal F(\Omega_1,\Omega_2,\dots,\Omega_N)=F\big(\lambda_1(\Omega_1),\dots,\lambda_1(\Omega_N)\big)$$
was initiated in \cite{bucve} and was then continued in \cite{benve} and \cite{stv}, where it was proved that if $d=2$ and if the $N$-uple  $\Omega_1,\dots,\Omega_N$ is a solution of 
\begin{equation}\label{e:multi}
\min\left\{\sum_{j=1}^N\Big(\lambda_1(\Omega_j)+\Lambda|\Omega_j|\Big)\ :\ \ \Omega_1,\dots,\Omega_N\ \text{ are disjoint quasi-open subsets of }\ D\right\},
\end{equation}
then each of the sets $\Omega_j$ has a $C^{1,\alpha}$ regular boundary. This result was recently extended to any dimension $d\ge 2$ in \cite{despve}. As in the one-phase $\mathcal F(\Omega)=\lambda_1(\Omega)$ and the vectorial $\ds\mathcal F(\Omega)=\sum_{j=1}^k\lambda_j(\Omega)$ problems, the crucial observation is that \eqref{e:multi} can be written (at least locally) as a minimization problem involving a single function, which changes sign. Precisely, in \cite{bucve} it was shown that one can reduce the analysis to the case of only two domains ($N=2$). Then, in \cite{stv} it was proved that if $u_1$ and $u_2$ are the first eigenfunctions of $\Omega_1$ and $\Omega_2$, then the function $u:=u_1-u_2$ is an almost-minimizer of the functional 
$$\int_D|\nabla u|^2\,dx+\Lambda\big|\{u\neq0\}\cap D\big|.$$
This allowed to prove the regularity of the free boundary for almost-minimizers in dimension two (see \cite{stv}) by means of the epiperimetric inequality from \cite{spve}. In higher dimension, the analysis was concluded in \cite{despve}, where was proved the regularity of both free boundaries 
$$\partial\{u>0\}\cap D\quad\text{and} \quad \partial\{u<0\}\cap D$$ 
in a neighborhood of $\partial\{u>0\}\cap\partial\{u<0\}$, for functions 
$u$ that solve a PDE in $\{u\neq 0\}$ and satisfy the following conditions on the boundary $\partial\{u\neq0\}\cap D$ in viscosity sense
\begin{equation}\label{e:opt_con_intro}
\begin{cases}
|\nabla u_+|=\alpha_+>0\quad\text{on}\quad \partial\{u>0\}\setminus \partial\{u<0\}\cap D,\\
|\nabla u_-|=\alpha_->0\quad\text{on}\quad \partial\{u<0\}\setminus \partial\{u>0\}\cap D,\\
|\nabla u_\pm|\ge \alpha_\pm\quad\text{and}\quad |\nabla u_+|^2-|\nabla u_-|^2=\alpha_+^2-\alpha_-^2\quad\text{on}\quad \partial\{u>0\}\cap \partial\{u<0\}\cap D.
\end{cases}
\end{equation}
As it can be easily seen from the analysis in \cite{stv}, this result applies directly to the multiphase problem \eqref{e:multi} by taking the constants $\alpha_+$ and $\alpha_-$ to be equal to $\sqrt{\Lambda}$. Unfortunately, the regularity theorem from \cite{despve} cannot be directly applied to \eqref{e:multi2} and \eqref{e:fun}. In fact, in the present work, a key point in the proof of \cref{thm:main} is to show that if $\Omega$ is an optimal set for \eqref{e:fun}, then there is a Lipschitz continuous (on the whole $\R^d$) second eigenfunction $u\in H^1_0(\Omega)$ that satisfies \eqref{e:opt_con_intro} in viscosity sense for some strictly positive constants $\alpha_+$ and $\alpha_-$. We will discuss the strategy of the proof in \cref{sub:plan}.

%, which are not explicitly related to $\Lambda$, but satisfy $\lambda_+^2+\lambda_-^2=\Lambda$.

%
%\smallskip
%
%%In the present paper we obtain the first result (\cref{thm:main}) that provides a complete description of the free boundary of optimal sets for a functional involving higher eigenvalues. 
%%%the non-differentiable functional $\mathcal F=\lambda_2$. 
%%We notice that, in our case ($k=2$), the aforementioned set of \emph{singular} branching points corresponds to the points where the two-phase free boundary meets the one-phase free boundary.
%
%The aim of the present paper is to give a complete description, including the branching points and the contact points with $\partial D$, of the boundary of the optimal sets for the second eigenvalue ($k=2$).

\subsection{Optimal quasi-open sets}\label{sub:intro-quasi-open}
The variational minimization problem \eqref{e:fun} is usually stated in the wider class of the so-called \it quasi-open \rm sets, as 
\begin{equation}\label{e:minFLambdaqo}
\min\Big\{\mathcal F_\Lambda(\Omega)\ :\ \Omega\subset D\;,\ \Omega\ \text{ quasi-open}\Big\}.
\end{equation}
%is usually stated in the wider class of the so-called \it quasi-open \rm sets $\Omega\subset D$. 
As explained in the previous section, the main reason is that a general theorem by Buttazzo and Dal Maso \cite{budm} provides the existence of optimal sets in this class for a large variety of functionals, which includes $\mathcal F_\Lambda$. Our regularity result holds also for minimizers in this class of sets. Before we state the result in this setting (\cref{thm:main2}), we briefly recall the main definitions in this context (for more details, we refer to the books \cite{evans-gariepy, hepi, bubu}). 
\subsubsection{Capacity}
The capacity of a set $E\subset \R^d$ is defined as
\begin{equation}\label{e:capacity}
\cp(E)=\inf\Big\{\int_{\R^d}\big(|\nabla u|^2+u^2\big)\,dx\ :\ u\in H^1(\R^d),\ u\ge 1\ \text{ in a neighborhood of }\ E\Big\}.
\end{equation}
It is well-known (see for instance \cite{evans-gariepy}) that any function $u\in H^1(\R^d)$, which by definition is defined almost-eveywhere in the sense of the Lebesgue measure, is also defined \it quasi-everywhere \rm on $\R^d$ in the following sense: there is a set $E_u\subset \R^d$ such that $\cp(E_u)=0$ and the limit
$$\lim_{r\to0}\mean{B_r(x_0)}{u(x)\,dx}\quad\text{exists for every}\quad x_0\in\R^d\setminus E_u.$$  
In particular, this allows to define $u$ pointwise everywhere on $\R^d\setminus E_u$ as%(notice that the definition does not depend on the choice of respresentative) as 
\begin{equation}\label{e:pw_def}
u(x_0):=\lim_{r\to0}\mean{B_r(x_0)}{u(x)\,dx}.
\end{equation}
Notice that the definition does not depend on the choice of respresentative of $u$ in $H^1(\R^d)$.

\subsubsection{Quasi-open sets and Sobolev spaces}
For every measurable set $\Omega\subset\R^d$ we define the space $H^1_0(\Omega)$ as 
$$H^1_0(\Omega)=\Big\{u\in H^1(\R^d)\ :\ \cp (\{u\neq 0\}\setminus \Omega)=0\Big\}.$$
When $\Omega$ is open, $H^1_0(\Omega)$ is precisely the closure of $C^\infty_c(\Omega)$ with respect to the $H^1$ norm (see for instance \cite{hepi}). When $\Omega$ is bounded, the embedding of $H^1_0(\Omega)$ in $L^2(\Omega)$ is compact.

We say that $\Omega$ is a quasi-open set if there is a function $u\in H^1(\R^d)$ satisfying \eqref{e:pw_def} outside a set of zero capacity and such that $\Omega=\{u>0\}$ up to a set of zero capacity; in particular, for every $u\in H^1(\R^d)$, the set $\Omega=\{u\neq 0\}$ is quasi-open and $u\in H^1_0(\Omega)$. 

Notice that a quasi-open set $\Omega$ is defined up to a set of zero capacity and that every open set is also quasi-open. Moreover, if $E$ is any subset of $\R^d$, then there is a unique (up to a set of zero capacity) quasi-open set  $\Omega$ such that 
$\cp(\Omega\setminus E)=0$ and $H^1_0(E)=H^1_0(\Omega)$. In other words, when we write $H^1_0(\Omega)$, we can always assume that $\Omega$ is quasi-open. 

\subsubsection{Spectrum of the Dirichlet Laplacian on quasi-open sets} Let $\Omega$  be a bounded quasi-open set in $\R^d$ and let $f\in L^2(\Omega)$. We say that $u\in H^1_0(\Omega)$ is a solution to 
$$-\Delta u=f\quad\text{in}\quad \Omega,$$
if for every $\varphi\in  H^1_0(\Omega)$, we have 
$$\int_{\Omega}\nabla u\cdot\nabla\varphi\,dx=\int_\Omega \varphi f\,dx.$$
The operator $\mathcal R_\Omega:L^2(\Omega)\to L^2(\Omega)$, that associates to each $f\in L^2(\Omega)$ the unique solution $u$ of the above equation, is linear, positive definite, compact and self-adjoint. Thus, its spectrum is discrete and made by eigenvalues that can be ordered in an infinitesimal and monotone decreasing sequence of positive real numbers. By definition, their inverse are the eigenvalues of the Dirichlet Laplacian on $\Omega$ and are denoted by $\lambda_k(\Omega)$, $k\in \N$,
$$0<\lambda_1(\Omega)\le \lambda_2(\Omega)\le\dots\le \lambda_k(\Omega)\le\dots$$
Moreover, there is a sequence of orthonormal (in $L^2(\Omega)$) eigenfunctions  $u_k\in H^1_0(\Omega)$, $k\in \N$, satisfying
$$-\Delta u_k=\lambda_k(\Omega)\,u_k\quad\text{in}\quad \Omega\,,\qquad \int_\Omega u_k^2\,dx=1.$$
Finally, we recall that for every $k\ge 1$, the eigenvalue $\lambda_k$ of the Dirichlet Laplacian can be obtained through the following min-max principle, 
\begin{equation}\label{e:var_lambda_k}
\lambda_k(\Omega):=\min_{E_k\subset H^1_0(\Omega)}\max_{u\in E_k\setminus \{0\}}\frac{\ds\int_\Omega|\nabla u|^2\,dx}{\ds\int_\Omega u^2\,dx},
\end{equation} 
where the minimum is taken over all $k$-dimensional linear subspaces $E_k$ of $H^1_0(\Omega)$.

\subsubsection{Existence of optimal quasi-open sets for~\eqref{e:minFLambdaqo}}\label{s:existenceqo}
This result follows from the Buttazzo-Dal Maso Theorem~\cite[Theorem~2.5]{budm}, but we provide a direct proof which is more suitable for our aims.
Let $\Omega_n$ be a minimizing sequence of quasi-open sets for $\mathcal F_\Lambda$ in $D$, that is 
$$\inf\Big\{\mathcal F_\Lambda(\Omega)\ :\ \Omega\subset D\ \text{quasi-open}\Big\}=\lim_{n\to\infty}\mathcal F_\Lambda(\Omega_n).$$
By the definition of $\lambda_2(\Omega_n)$, there are functions $u_n$ and $v_n$ in $H^1_0(\Omega_n)$ such that 
$$\int_{D}|\nabla u_n|^2\,dx= \lambda_2(\Omega_n)\,,\quad \int_{D}|\nabla v_n|^2\,dx\le \lambda_2(\Omega_n)\,,\quad\int_{D} u_n^2\,dx=\int_{D} v_n^2\,dx=1\quad\text{and}\quad  \int_{D} u_nv_n\,dx=0.$$  
Moreover, we can assume that $\Omega_n=\{u_n^2+v_n^2>0\}.$ Since the sequences $u_n$ and $v_n$ are uniformly bounded in $H^1_0(D)$, up to a subsequence, we can assume that 
$u_n$ (resp.\,$v_n$) converges to a function $u_\infty$ (resp.\,$v_\infty$) weakly in $H^1_0(D)$, strongly in $L^2(D)$ and pointwise almost everywhere. Now, by the semicontinuity of the $H^1$ norm, we have 
\[
\begin{split}
\int_{D}|\nabla u_\infty|^2\,dx\le \liminf_{n\to+\infty}\lambda_2(\Omega_n)\,,\quad \int_{D}|\nabla v_\infty|^2\,dx\le \liminf_{n\to+\infty}\lambda_2(\Omega_n)\,,\\
\int_{D} u_\infty^2\,dx=\int_{D} v_\infty^2\,dx=1\quad\text{and}\quad  \int_{D} u_\infty v_\infty\,dx=0,
\end{split}
\]
where $\Omega_\infty$ is the set $\{u_\infty^2+v_\infty^2>0\}$. Thus, 
$$\lambda_2(\Omega_\infty)\le \liminf_{n\to\infty}\lambda_2(\Omega_n).$$
On the other hand, the pointwise convergence of $u_n$ and $v_n$ gives that 
$$\ind_{\Omega_\infty}\le \liminf_{n\to\infty}\ind_{\Omega_n},$$
and by the Fatou's Lemma, we get 
$$|\Omega_\infty|\le \liminf_{n\to\infty}|\Omega_n|.$$
Thus, we obtain
$$\mathcal F_\Lambda(\Omega_\infty)\le \liminf_{n\to\infty}\mathcal F_\Lambda(\Omega_n)\le \inf\Big\{\mathcal F_\Lambda(\Omega)\ :\ \Omega\subset D\ \text{quasi-open}\Big\},$$
which proves that $\Omega_\infty$ is an optimal quasi-open set.

\subsubsection{Regularity of the optimal quasi-open set}
A regularity result, analogous to \cref{thm:main}, holds also for the minimizers of $\mathcal F_\Lambda$ among quasi-open sets. In fact, the two results are equivalent (see \cref{subsub:equivalence}).
 
 \begin{theorem}\label{thm:main2}
 	Let $D\subset \R^d$ be an open  bounded set of class $C^{1,\beta}$ for some $\beta>0$, and let $\Lambda>0$ be a given constant. 
 	Let $\Omega\subset D$ be a quasi-open set that minimizes $\mathcal F_\Lambda$ in $D$, that is, 
 \begin{equation}\label{e:min-pb-quasi-open}
 \mathcal F_\Lambda(\Omega)\le \mathcal F_\Lambda(\widetilde\Omega)\quad\text{for every quasi-open set}\quad \widetilde\Omega\subset D.
 \end{equation}
 	Then, there are two disjoint open sets $\Omega_+$ and $\Omega_-$ such that: 
 	$$\cp\big((\Omega_+\cup\Omega_-)\setminus\Omega\big)=0\ ,\qquad\lambda_2(\Omega_+\cup\Omega_-)=\lambda_2(\Omega)\qquad \text{and}\qquad |\Omega\setminus (\Omega_+\cup\Omega_-)|=0.$$
 The boundaries $\partial \Omega_+$ and $\partial\Omega_-$ can be decomposed as the disjoint union of a regular and a singular part
 	$$\partial \Omega_\pm=\text{\rm Reg}\,(\partial \Omega_\pm)\cup \text{\rm Sing}\,(\partial \Omega_\pm),$$
 for which the claims {\rm (i)} and {\rm (ii)} of \cref{thm:main} hold.
 \end{theorem}

\subsubsection{Equivalence of \cref{thm:main} and \cref{thm:main2}}\label{subsub:equivalence}
We will first show that \cref{thm:main2} implies \cref{thm:main}. Let $\Omega\subset D$ be an open set satisfying \eqref{e:min-pb-open}. We will prove that it satisfies \eqref{e:min-pb-quasi-open}. We will use the fact that if $\widetilde\Omega\subset D$ is any quasi-open set, then there is a sequence of open sets $\omega_n$ such that
$$\ds\lim_{n\to\infty}\cp(\omega_n)=0\qquad \text{and}\qquad \widetilde\Omega\cup\omega_n\,\text{ is open for every  }\,n\in\N.$$
%\begin{itemize}
%\item $\widetilde\Omega\cup\omega_n$ is open for every $n\in\N$, 
%\item $\ds\lim_{n\to\infty}\cp(\omega_n)=0$.
%\end{itemize}
In particular, the sets  $\widetilde\Omega_n:=\widetilde\Omega\cup\big(\omega_n\cap D\big)$ are open and satisfy
$$\lambda_2(\widetilde\Omega_n)\le \lambda_2(\widetilde\Omega)\qquad\text{and}\qquad \lim_{n\to\infty}|\widetilde\Omega_n|=|\widetilde\Omega|.$$
The first inequality follows directly from \eqref{e:lambda_2}, while the second claim follows from the fact that $|\omega_n|\le \cp(\omega_n)$, which is a consequence of \eqref{e:capacity}.
Now, since $\Omega$ satisfies  \eqref{e:min-pb-open}, we have that $\mathcal F_\Lambda(\Omega)\le \mathcal F_\Lambda(\widetilde\Omega_n)$, which gives
$$\mathcal F_\Lambda(\Omega)\le \liminf_{n\to\infty}\mathcal F_\Lambda(\widetilde\Omega_n)\le \mathcal F_\Lambda(\widetilde\Omega),$$
and proves that  $\Omega$ is also a solution to \eqref{e:min-pb-quasi-open}.

Conversely, we also have that \cref{thm:main} implies \cref{thm:main2}. This is a consequence of \cite[Corollary 5.11]{bucve}, which states that if $\Omega_{qo}$ is a quasi-open set that satisfies \eqref{e:min-pb-quasi-open}, then there exists an open set 
$\Omega_{o}\subset D$ such that $\Omega_o\subset\Omega_{qo}$ (in the sense that $\cp(\Omega_{qo}\setminus\Omega_o)=0$) and is a solution to \eqref{e:min-pb-quasi-open} (and thus, also to \eqref{e:min-pb-open}).

\subsection{Plan of the paper and outline of the proof of \cref{thm:main2}}\label{sub:plan}
%
%First of all, in \cref{sub:history} we recall the basic notions about eigenvalues of the Dirichlet Laplacian and the main results (with a focus on regularity properties) about their minimization.
%Then \cref{sub:intro-quasi-open} is devoted to the formulation of the problem in the class of quasi-open sets and to the proof of the equivalence between the two main \cref{thm:main} and \cref{thm:main2}.
%The existence of a minimizer for problem~\eqref{eq:main00} in the class of quasi-open sets follows directly from Buttazzo-Dal Maso Theorem~\cite{budm}.

In \cref{s:equivalencefb} we show that \eqref{e:minFLambdaqo} is equivalent (in some suitable sense) to the variational free boundary problem 
%$$\max\left\{\int_D |\nabla u_+|^2\,dx\ ;\ \int_D |\nabla u_-|^2\,dx\right\}+\Lambda\big|\{u\neq 0\}\big|.$$
\begin{equation}\label{e:fb_infty_intro}
\min\left\{J_\infty(v_+,v_-)+\Lambda\big|\{v\neq 0\}\big|\ :\ v\in H^1_0(D),\ \int_D v_+^2\,dx=\int_D v_-^2\,dx=1\right\},
\end{equation}
where 
$$J_\infty(v_+,v_-)=\max\left\{\int_D |\nabla v_+|^2\,dx\ ;\ \int_D |\nabla v_-|^2\,dx\right\}.$$

In \cref{s:nondegeneracy} we prove the non-degeneracy result for minimizers $u$ of~\eqref{e:fb_infty_intro}. This nondegeneracy, together with the three-phase monotonicity formula from \cite{mono} implies that the two-phase free boundary $\partial\{u>0\}\cap\partial\{u<0\}$ does not touch $\partial D$. As a consequence of this and of the (one-phase) regularity result from \cite{rtv}, we obtain that the one-phase free boundaries $\partial \Om_\pm\setminus \partial \Om_\mp$ are $C^{1,\alpha}$ regular in a neighborhood of the contact set $\partial \Om_\pm\cap\partial D$.

In \cref{sect:lip} we prove that if $\Omega$ is a solution to \eqref{e:minFLambdaqo}, then there is a Lipschitz continuous function $u:D\to\R$, which is a sign-changing eigenfunction on $\Omega$ and (after an appropriate rescaling of the positive and negative parts) a minimizer of \eqref{e:fb_infty_intro}. In \cref{sect:approx}, we prove that this function $u$ satisfies a first order optimality condition.

In \cref{sect:blowup} we first use a Weiss-type monotonicity formula to prove the homogeneity of blow-up limits. Using this, we are able to classify the two-phase blow-up limits in \cref{t:main_intext}. Finally, using the result from \cite{despve}, in  \cref{c:main_cor} we prove that both $\partial \Om_+$ and $\partial \Om_-$ are $C^{1,\alpha}$ regular in a neighborhood of the two-phase free boundary $\partial \Om_+\cap \partial \Om_-$.

\medskip
{\bf Notation.} 
For the whole paper $d\geq 2$ is an integer and denotes the dimension of the space.
We use the notation for the positive and negative part of a function:
\[
v_+=\max\{v,0\}\qquad\text{and}\qquad v_-:=\max\{-v,0\},
\]
and if the function has already a subscript, such as $v_i$, then we use the notation\[
v_i^+=\max\{v_i,0\}\qquad\text{and}\qquad v_i^-:=\max\{-v_i,0\}.
\]
Given a function $u\colon \R^N\to \R$, we define \[
\Om_u:=\{u\not=0\},\qquad \Om_u^+:=\{u>0\},\qquad \Om_u^-:=\{u<0\},
\]
and if a function $u\in H^1_0(D)$ for some domain $D\subset \R^N$, we implicitly extend $u$ to zero outside $D$.

\section{Preliminary facts about the principal eigenfunctions on quasi-open sets}\label{sub:prelimiary_first_eigenfunctions}

In this section we recall some basic properties of the principal eigenfunctions on quasi-open sets, which we will use several times in the paper. Throughout this section,  we consider a quasi-open set $\Omega\subset\R^d$ of finite measure and a first eigenfunction $u$ of the Dirichlet Laplacian on $\Omega$, that is,  $u\in H^1_0(\Omega)$ is a non-negative minimizer of 
$$\lambda_1(\Omega)=\Big\{\int_\Omega|\nabla u|^2\,dx\ :\ u\in H^1_0(\Omega),\ \int_\Omega u^2\,dx=1\Big\}.$$
We suppose that $u$ is extended as $0$ outside $\Omega$ and that  $u\ge 0$ almost everywhere in $\R^d$. 

\subsection{Subharmonicity and global $L^\infty$ bound}\label{oss:subharmonicity}
%Let $\Omega$ be a quasi-open set of finite measure in $\R^d$ and let $u$ be a first eigenfunction of the Dirichlet Laplacian on $\Omega$, that is, $u$ is a non-negative minimizer of 
%$$\lambda_1(\Omega)=\Big\{\int_\Omega|\nabla u|^2\,dx\ :\ u\in H^1_0(\Omega),\ \int_\Omega u^2\,dx=1\Big\}.$$
%We claim that there is a constant $C>0$, depending only on $d$ and $\lambda_1(\Omega)$, such that
%\begin{equation}\label{e:subharmonicity}
%\Delta u+C\ge0\quad\text{in sense of distributions in } \R^d.
%\end{equation}
We first notice that $u$ is a (weak) solution of the PDE
$$\Delta u+\lambda_1(\Omega)\,u=0\quad\text{in}\quad \Omega.$$
Moreover, since $u$ is nonnegative, a standard argument (see for instance \cite[Lemma 2.7]{velectures}) proves that
$$\Delta u+\lambda_1(\Omega)\,u\ge 0\quad\text{in sense of distributions in } \R^d.$$
Precisely, for every non-negative function $\varphi\in C^\infty_c(\R^d)$ (notice that one can take also $\varphi\in H^1(\R^d)$),
$$\int_{\R^d}\Big(-\nabla u\cdot\nabla\varphi+\lambda_1(\Omega)\,u\,\varphi\Big)\,dx\ge 0.$$
Now, we recall that the supremum of the eigenfunction can be estimated only in terms of the associated eigenvalue. Indeed, there is a dimensional constant $C_d>0$ (see for instance  \cite[Example~2.1.8]{davies} or \cite[Proposition~3.4.37]{vetesi}) such that 
$$\|u\|_{L^\infty(\R^d)}\le C_d\big(\lambda_1(\Omega)\big)^{\sfrac{d}4}.$$ 
As a consequence, we get that 
\begin{equation}\label{e:subharmonicity_better}
\Delta u+C_d\big(\lambda_1(\Omega)\big)^{\sfrac{(d+4)}4}\ge 0\quad\text{in sense of distributions in } \R^d.
\end{equation}

\subsection{Pointwise definition and local $L^\infty$ bound} Let now $x_0\in\R^d$ be any point. By \eqref{e:subharmonicity_better}, the function 
$$u_{x_0}(x):= u(x)+C_d\big(\lambda_1(\Omega)\big)^{\sfrac{(d+4)}4}\,\frac{|x-x_0|^2}{2d}.$$
is subharmonic in $\R^d$ and so, the limit 
$$\lim_{r\to0}\mean{B_r(x_0)}{u_{x_0}(x)\,dx},$$
exists. Now, since by construction $\|u-u_{x_0}\|_{L^\infty(B_r(x_0))}\le Cr^2$, we also get 
$$\lim_{r\to0}\mean{B_r(x_0)}{u(x)\,dx}=\lim_{r\to0}\mean{B_r(x_0)}{u_{x_0}(x)\,dx}.$$ 
%$$\lim_{r\to0}\|u-u_{x_0}\|_{L^\infty(B_r(x_0))}=0,$$
Thus, we can choose a representative of $u$ which is defined everywhere in $\R^d$ (recall that $u\in H^1(\R^d)$ is an equivalence class in $L^2(\R^d)$). Precisely, from now on, we will always assume that 
$$u(x_0)=\lim_{r\to0}\mean{B_r(x_0)}{u(x)\,dx}\quad\text{for every}\quad x_0\in\R^d.$$
Finally, as another consequence of the subharmonicity of $u_{x_0}$, we obtain that, for every $0<\sigma<1$ and every $r>0$,  the following estimate holds
\begin{equation}\label{e:L2-Linfty}
\|u\|_{L^\infty(B_{\sigma r}(x_0))}\le \frac{1}{(1-\sigma)^d}\mean{\partial B_r(x_0)}{u\,d\HH^{d-1}}+C_d\big(\lambda_1(\Omega)\big)^{\sfrac{(d+4)}4}r^2.
\end{equation}

\section{Equivalent formulations of the shape optimization problem}\label{s:equivalencefb}

\subsection{A variational free boundary problem}
Let $\Omega$ be a quasi-open set in $\R^d$. 
Then, we can give an equivalent formulation of $\lambda_2(\Omega)$ in terms of a two-phase free boundary problem in $\Omega$. Precisely, we have the following lemma.
\begin{lemma}[Second eigenvalue and optimal partitions of a fixed domain]\label{l:equiv-def-lambda-2}
Let $\Omega$ be a bounded open  (or quasi-open) set in $\R^d$. Then, 
\begin{equation}\label{e:lambda2-J}
\lambda_2(\Omega):=\min\left\{J_\infty(v_+,v_-)\ :\ v\in H^1_0(\Omega),\ \int_\Omega v_+^2\,dx=\int_\Omega v_-^2\,dx=1\right\},
\end{equation}
where the functional $J_\infty: H^1_0(\Omega)\times H^1_0(\Omega)\to\R$ is defined by 
\begin{equation}\label{e:J_infty}
J_\infty(v_+,v_-):=\max\left\{\int_\Omega |\nabla v_+|^2\,dx\ ;\ \int_\Omega |\nabla v_-|^2\,dx\right\}.
\end{equation}
\end{lemma}
\begin{proof}
We first notice that by the compactness of the embedding $H^1_0(\Omega)$ into $L^2(\Omega)$, there is a function $$u=u_+-u_-\in H^1_0(\Omega)$$ that realizes the minimum in \eqref{e:lambda2-J}, that is $\ds\int_\Omega u_+^2\,dx=\int_\Omega u_-^2\,dx=1 $, and 
$$J_\infty(u_+,u_-)=\min\left\{J_\infty(v_+,v_-)\ :\ v\in H^1_0(\Omega),\ \int_\Omega v_+^2\,dx=\int_\Omega v_-^2\,dx=1\right\}.$$
Now, since the space generated by $u_+$ and $u_-$ is a two-dimensional subspace of $H^1_0(\Omega)$, we get that 
$$\lambda_2(\Omega)\le J_\infty(u_+,u_-).$$
On the other hand, let $u_1$ and $u_2$ be one first and one second eigenfunction of the Dirichlet Laplacian on $\Omega$. Then, we have that
$$\int_\Omega u_1^2\,dx=\int_\Omega u_2^2\,dx=1\qquad\text{and}\qquad \int_\Omega u_1u_2\,dx=0,$$
$$\int_\Omega|\nabla u_1|^2\,dx=\lambda_1(\Omega)\le \lambda_2(\Omega)=\int_\Omega|\nabla u_2|^2\,dx,$$
$$-\Delta u_1=\lambda_1(\Omega)u_1\quad\text{in}\quad \Omega,$$
\begin{equation}\label{e:eq-lambda-2}
-\Delta u_2=\lambda_2(\Omega)u_2\quad\text{in}\quad \Omega.
\end{equation}
In particular, the space $V\subset H^1_0(\Omega)$ generated by $u_1$ and $u_2$ realizes the minimum in \eqref{e:lambda_2}. We now consider two cases. First, if $u_2$ changes sign, then we define the functions
$$\varphi_+:=\left(\int_\Omega (u_2^+)^2\,dx\right)^{-1/2}u_2^+\qquad\text{and}\qquad \varphi_-:=\left(\int_\Omega (u_2^-)^2\,dx\right)^{-1/2}u_2^-.$$
By testing the equation \eqref{e:eq-lambda-2} with $\varphi_+$ and $\varphi_-$, we get that 
$$\int_\Omega|\nabla \varphi_+|^2\,dx=\lambda_2(\Omega)=\int_\Omega|\nabla \varphi_-|^2\,dx.$$
Thus, 
$$J_\infty(u_+,u_-)\le J_\infty(\varphi_+,\varphi_-)=\lambda_2(\Omega),$$
which concludes the proof of \eqref{e:lambda2-J} in the case when $u_2$ changes sign. Moreover, by the same argument, we get that if $u_1$ changes sign, then $\lambda_1(\Omega)=\lambda_2(\Omega)$ and \eqref{e:lambda2-J} holds. Suppose now that $u_2\ge 0$ and $u_1\ge 0$. Then, the orthogonality in $L^2(\Omega)$ implies that they have disjoint supports and that, by taking $\psi=u_2-u_1$, we have that 
$$J_\infty(u_+,u_-)\le J_\infty(\psi_+,\psi_-)=\max\{\lambda_1(\Omega),\lambda_2(\Omega)\}=\lambda_2(\Omega),$$
which concludes the proof.
\end{proof}	
As a consequence, we can reformulate \eqref{e:minFLambdaqo} as a variational free boundary problem for the functional $J_\infty$
\begin{equation}\label{e:fb_infty}
\min\left\{J_\infty(v_+,v_-)+\Lambda\big|\{v\neq 0\}\big|\ :\ v\in H^1_0(D),\ \int_D v_+^2\,dx=\int_D v_-^2\,dx=1\right\}.
\end{equation}
We will prove that these two problems are equivalent in  \cref{p:equivalence1} and \cref{p:equivalence2}. In the proofs we will use several times the following simple fact. 
\begin{lemma}\label{l:lambda_1_continuity}
	Suppose that $\Omega$ is a bounded quasi-open set in $\R^d$, $d\ge 2$, and let $x_0\in\R^d$. 
	%	and let $f:(0,+\infty)\to\R$ be the function 
	%	$$f(r)=\lambda_1\big(\Omega\setminus \overline B_r(x_0)\big).$$	
	Then
	$$\lim_{r\to0^+}\lambda_1\big(\Omega\setminus \overline B_r(x_0)\big)=\lambda_1(\Omega).$$
\end{lemma}	
\begin{proof}
	Assume that $d\ge 3$, the case $d=2$ being analogous. Let $u$	be the first (normalized) eigenfunction on $\Omega$ and let $\phi_r:\R^d\to[0,1]$ be the function
	$$\phi_r=1\quad\text{in}\quad\R^d\setminus B_{2r}(x_0),\qquad \phi_r=0\quad\text{in}\quad B_{r}(x_0),\qquad  \phi_r=\frac1r(|x|-r)\quad\text{in}\quad B_{2r}(x_0)\setminus B_{r}(x_0).$$
	Since $\lambda_1(\Omega)\le \lambda_1(\Omega\setminus \overline B_r(x_0))$, we only have to bound $\lambda_1(\Omega\setminus \overline B_r(x_0))$ from above:
	$$\lambda_1(\Omega\setminus \overline B_r(x_0))\le \frac{\int |\nabla(u\phi_r)|^2\,dx}{\int (u\phi_r)^2\,dx}\le \Big(1-\int_{B_{2r}}u^2\,dx\Big)^{-1}\Big(\lambda_1(\Omega)+2\sqrt{\lambda_1(\Omega)}\|\nabla\phi_r\|_{L^2}+\|u\|_{L^\infty}^2\|\nabla\phi_r\|_{L^2}^2\Big).$$
	Passing to the limit as $r\to0$, we get the claim. 
\end{proof}	

\begin{proposition}\label{p:equivalence1}
Let $D$ be a bounded open set in $\R^d$ and let $\Lambda>0$ be a given constant. 
\begin{enumerate}[\rm(i)]
\item If $\Omega\subset D$ is a quasi-open set that satisfies \eqref{e:min-pb-quasi-open} and if $u_2\in H^1_0(\Omega)$ is a sign-changing second eigenfunction of the Dirichlet Laplacian on $\Omega$, then the function $u:=u_+-u_-$ defined by 
$$u_+:=\left(\int_\Omega (u_2^+)^2\,dx\right)^{-1}u_2^+\qquad\text{and}\qquad u_-:=\left(\int_\Omega (u_2^-)^2\,dx\right)^{-1}u_2^-.$$
is a solution to \eqref{e:fb_infty}. 
\item If $\Omega\subset D$ is a quasi-open set that satisfies \eqref{e:min-pb-quasi-open} and if $u_2\in H^1_0(\Omega)$ is a nonnegative and normalized second eigenfunction of the Dirichlet Laplacian on $\Omega$, then $\lambda_1(\Omega)=\lambda_2(\Omega)$ and there exists another nonnegative and normalized eigenfunction $u_1$ (corresponding to the eigenvalue $\lambda_1(\Omega)=\lambda_2(\Omega)$) orthogonal to $u_2$ in $L^2(D)$, such that $u:=u_2-u_1$ 
is a solution to \eqref{e:fb_infty}. 
\end{enumerate}
\end{proposition}	
\begin{proof}
We first notice that, by the definition of $\lambda_2$, if the function $v\in H^1_0(D)$ is such that 
$$\int_D v_+^2\,dx=\int_D v_-^2\,dx=1,$$
then $\lambda_2(\{v\neq 0\})\le J_\infty(v_+,v_-)$. Now, if $u$ is the function from {\rm(i)}, then 
$$J_\infty(u_+,u_-)+\Lambda|\Omega_u|=\mathcal F_\Lambda(\Omega_u)\le\mathcal F_\Lambda(\Omega_v)\le J_\infty(v_+,v_-)+\Lambda|\Omega_v|,$$
where $\Omega_u=\{u\neq0\}$ and $\Omega_v=\{v\neq0\}$.
This proves {\rm(i)}.\medskip 

Let now $u_1$ and $u_2$ be as in {\rm (ii)}. Then, 
$$\int_{D}|\nabla u_2|^2\,dx=\lambda_2(\Omega)\,,\qquad\int_{D}|\nabla u_1|^2\,dx=\lambda_1(\Omega)\,,\qquad \int_{D} u_1^2\,dx=\int_{D}u_2^2\,dx=1\,.$$
Now, suppose that $\lambda_1(\Omega)<\lambda_2(\Omega)$. We pick a point $x_0$ of Lebesgue density $1$ for the set $\{u_1>0\}$ and consider the set 
$$\Omega_r:=\{u_2>0\}\cup\Big(\{u_1>0\}\setminus \overline B_r(x_0)\Big).$$ 
By \cref{l:lambda_1_continuity}, we get that for $r$ small enough 
$$\lambda_1\big(\{u_2>0\}\big)=\lambda_2(\Omega)>\lambda_1\big(\{u_1>0\}\setminus {\overline B_r(x_0)}\big)>\lambda_1\big(\{u_1>0\}\big)=\lambda_1(\Omega).$$
In particular, this implies that 
$$\lambda_2(\Omega_r)= \lambda_2(\Omega),$$
while on the other hand $|\Omega_r|<|\Omega|$, which contradicts the minimality of $\Omega$. This implies that $\lambda_1(\Omega)=\lambda_2(\Omega)$ and the claim now follows as in the proof of {\rm (i)}.
\end{proof}	

\begin{proposition}\label{p:equivalence2}
	Let $D$ be a bounded open set in $\R^d$ and let $\Lambda>0$. Suppose that the function $u\in H^1_0(D)$ is a solution to \eqref{e:fb_infty}. Then, 
		\begin{equation}\label{e:long_equation_0}
	\int_D|\nabla u_+|^2\,dx=\int_D|\nabla u_-|^2\,dx\,.
	\end{equation}
	Moreover, setting 
	$$\Omega_+=\{u>0\}\,,\quad \Omega_-=\{u<0\}\quad\text{and}\quad \Omega=\Omega_+\cup\Omega_-\ ,$$
	we have that:  
%the couple of quasi-open sets $\big(\Omega_+^\ast,\Omega_-^\ast\big)$ is a solution of \eqref{e:minpb-lambda-1-partition}, then:
	\begin{enumerate}[\rm (i)]
		\item $u_+$ is the first eigenfunction on $\Omega_+$ and $u_-$ is the first eigenfunction on $\Omega_-$, that is,  
\begin{equation}\label{e:long_equation}
\int_D|\nabla u_+|^2\,dx=\lambda_1(\Omega_+)\qquad\text{and}\qquad\int_D|\nabla u_-|^2\,dx=\lambda_1(\Omega_-)\,.
\end{equation}
\item The set $\Omega$ minimizes \eqref{e:min-pb-quasi-open} and 
$$\lambda_1(\Omega_+)=\lambda_1(\Omega_-)=\lambda_2(\Omega).$$
\item There are constants $a>0$ and $b>0$ such that the function $u_2=au_+-bu_-$ is a second eigenfunction on $\Omega$, that is, 
$$-\Delta u_2=\lambda_2(\Omega)u_2\quad\text{in}\quad\Omega.$$
\item The sets $\Omega_+$ and $\Omega_-$ are inward minimizing for the functional $\lambda_1+\Lambda|\cdot|$, that is, 
\begin{equation}\label{e:inwards0}
\lambda_1(\Omega_\pm)+\Lambda|\Omega_\pm|\le \lambda_1(\widetilde\Omega)+\Lambda|\widetilde\Omega|\quad\text{for every quasi-open set}\quad \widetilde\Omega\subset\Omega_\pm\,.
\end{equation}	
\item Setting $c_+=|\Omega_+|$ and $c_-=|\Omega_-|$, we have
\begin{equation}\label{eq:minla1}
\begin{split}
\lambda_1(\Omega_+)=\min\Big\{\lambda_1(A) :  A\subset D \text{ quasi-open },\ |A\cap \Omega_-|=0,\ |A|=c_+\Big\},\qquad\qquad\qquad\\
\qquad\qquad\qquad\lambda_1(\Omega_-)=\min\Big\{\lambda_1(A) :  A\subset D \text{ quasi-open },\ |A\cap \Omega_+|=0,\ |A|=c_-\Big\}.\qedhere
\end{split}
\end{equation}	
\end{enumerate}	
\end{proposition}
\begin{proof}
The first claim \eqref{e:long_equation_0} follows as in the proof of \cref{l:lambda_1_continuity}; in fact, if the Dirichlet energy of $u_-$ is smaller than the one of $u_+$, then we can construct a competitor of the form $u_+-\phi_ru_-$ with the same energy
$$J_\infty(u_+,\phi_ru_-)=J_\infty(u_+,u_-),$$ 
but with smaller support. The claim {\rm(i)} now follows directly from the definition of $J_\infty$. \medskip

\noindent In order to prove {\rm(ii)} suppose that $\Omega^\ast$ is a solution to \eqref{e:min-pb-quasi-open}. 
Then, by \cref{p:equivalence1}, there is a second eigenfunction $u^\ast\in H^1_0(\Omega^\ast)$, corresponding to $\lambda_2(\Om^\ast)=J_\infty(u^*_+,u^*_-)$ with \[
\int_D(u^*_+)^2=\int_D(u^*_-)^2=1.
\]  
Thus, the minimality of $u$ gives that 
$$J_\infty(u^\ast_+,u^\ast_-)+\Lambda|\Omega^\ast|\ge J_\infty(u_+,u_-)+\Lambda|\Omega|.$$
On the other hand, the minimality of $\Omega^\ast$ implies that 
$$\lambda_2(\Omega^\ast)+\Lambda|\Omega^\ast|\le \lambda_2(\Omega)+\Lambda|\Omega|,$$
and we can combine these inequalities to get {\rm (ii)}. \medskip
 
 In order to prove {\rm(iii)}, we consider two cases. First, if $\lambda_1(\Omega)=\lambda_2(\Omega)$, then both the functions $u_+$ and $u_-$ are first eigenfunctions on $\Omega$ and so the equations 
 $$-\Delta u_+=\lambda_2(\Omega)u_+\qquad\text{and}\qquad -\Delta u_-=\lambda_2(\Omega)u_-\,$$
 hold in the entire domain $\Omega$, that is, the two equations hold weakly in $H^1_0(\Omega)$: in particular, this proves the claim. Second, we consider the case $\lambda_1(\Omega)<\lambda_2(\Omega)$ and we choose a non-negative eigenfunction $u_1$ corresponding to $\lambda_1(\Omega)$. Since $\lambda_1(\Omega)<\lambda_1(\Omega_\pm)$, we have that $\{u_1>0\}$ intersects both $\Omega_+$ and $\Omega_-$. In particular, we can find constants $a$ and $b$ such that the function $u_2:=au_+-bu_-$ is such that:
 $$\int_\Omega u_2^2\,dx=1\,,\qquad \int_\Omega u_2u_1\,dx=0\qquad\text{and}\qquad \int_\Omega|\nabla u_2|^2\,dx=\lambda_2(\Omega).$$
 As a consequence, using the variational formulation \eqref{e:lambda_2} and comparing the space generated by the couple $(u_1,u_2)$ with the spaces generated by $(u_1,u_2+\eps\phi)$ for $\phi\in H^1_0(\Omega)$ and $\eps$ small, we obtain that $u_2$ is in fact a second eigenfunction corresponding to the eigenvalue $\lambda_2(\Omega)$ : 
 $$-\Delta u_2=\lambda_2(\Omega)u_2\quad\text{in}\quad \Omega.$$
The claim {\rm (iv)} is an immediate consequence of testing in~\eqref{e:fb_infty} the optimality of the function $u$ with the functions $u_+-\widetilde u_-$ and $\widetilde u_+-u_-$, where $\widetilde u_+$ and $\widetilde u_-$ are normalized first eigenfunctions on $\widetilde \Omega_+$ and $\widetilde\Omega_-$.\medskip
 
 We finally deal with {\rm (v)}. Suppose by contradiction that there is a set $\widetilde \Omega$  such that 
 $$\widetilde\Omega\subset D\,,\quad |\widetilde \Omega\cap\Omega_-|=0\,,\quad |\widetilde \Omega|=c_+\,\quad\text{and}\quad\, \lambda_1(\widetilde \Omega)<\lambda_1(\Omega_+).$$ 
 Then, pick a point $x_0$ of density one for $\widetilde\Omega$ and a sufficiently small radius $r>0$ such that 
 $$\lambda_1(\Omega_+)>\lambda_1(\widetilde \Omega\setminus \overline B_r(x_0))\ge \lambda_1(\widetilde\Omega),$$
 and let $\widetilde u_+$ be the first eigenfunction on $\widetilde \Omega\setminus \overline B_r(x_0)$. Then, 
 $$J_\infty(\widetilde u_+,u_-)+|\widetilde \Omega\setminus \overline B_r(x_0)|+|\Omega_-|=\lambda_1(\widetilde \Omega\setminus \overline B_r(x_0))+|\widetilde \Omega\setminus \overline B_r(x_0)|+|\Omega_-|< \lambda_1(\Omega_+)+|\Omega_+|+|\Omega_-|\ ,$$
 which contradicts the minimality of $u$.
\end{proof}

\section{Inwards minimizing property, non-degeneracy and two-phase points}\label{s:nondegeneracy}
	
\begin{lemma}[Nondegeneracy]\label{l:nondegeneracy}
For every pair of constants $C>0$ and $\Lambda>0$, there are constants $r_0>0$ and $\eta>0$, depending on $C$, $\Lambda$ and the dimension $d$, such that the following holds. Suppose that the bounded quasi-open set $\Omega\subset \R^d$ is such that:
\begin{itemize}
\item $\lambda_1(\Omega)\le C$ ;
\item $\Omega$ satisfies the inwards minimizing property 
\begin{equation}\label{e:inwards}
\lambda_1(\Omega)+\Lambda|\Omega|\le \lambda_1(\widetilde\Omega)+\Lambda|\widetilde\Omega|\quad\text{for every quasi-open set}\quad \widetilde\Omega\subset\Omega\,,
\end{equation}	
%\item $\|u\|_{L^\infty(B_r(x_0))}\le \eta r$, where $u$ is the first eigenfunction on $\Omega$. 
\item $\ds\mean{\partial B_r(x_0)}{u\,d\HH^{d-1}}\le \eta r$, for $r\leq r_0$ and where $u$ is the first eigenfunction on $\Omega$. 
\end{itemize}
Then $u=0$ in $B_{\sfrac{r}2}(x_0)$.
\end{lemma}	
This is a well-known result, for a proof see for example~\cite{bucve}.

As a consequence of \cref{l:nondegeneracy}, we have the following result. 
We use the notation $\mathcal C_\delta$ for the cone 
	$$\mathcal C_\delta:=\big\{x\in\R^d\ :\ x_d>\delta |x|\big\}.$$

\begin{proposition}[Triple points]\label{p:triple}
	Suppose that $\Omega_+$ and $\Omega_-$ are disjoint bounded quasi-open sets in $\R^d$ each one satisfying the inwards minimizing property 
	\begin{equation}
	\lambda_1(\Omega_\pm)+\Lambda|\Omega_\pm|\le \lambda_1(\widetilde\Omega)+\Lambda|\widetilde\Omega|\quad\text{for every quasi-open set}\quad \widetilde\Omega\subset\Omega_\pm\,,
	\end{equation}	
	for some $\Lambda>0$. Then, there is a constant $\delta>0$ such that if 
	$$\Omega_+\cap\Omega_-\cap B_R=\emptyset\,,\qquad \mathcal C_\delta\cap \Omega_+\cap B_R=\emptyset \qquad\text{and}\qquad \mathcal C_\delta\cap \Omega_-\cap B_R=\emptyset\,,$$
	for some $R>0$, then there exists $\eps>0$ such that 
	$$\Omega_+\cap B_\eps=\emptyset\qquad\text{or}\qquad \Omega_-\cap B_\eps=\emptyset\,.$$
\end{proposition}	

In the proof of \cref{p:triple}, we will use the following two lemmas.
 
\begin{lemma}[Three-phase monotonicity formula \cite{mono,bucve}]\label{l:three-phase}
	Let $u_i \in H^1(B_1)$, $i=1,2,3$, be three non-negative functions such that:
	\begin{itemize}
	\item $\Delta u_i+1\ge0$ in $B_1$ in sense of distributions, for every $i=1,2,3$;
	\item $\ds\int_{\R^d} u_iu_j\,dx=0$, for every pair $i\neq j\in\{1,2,3\}$. 
	\end{itemize}
Then there are dimensional constants $\eps>0$ and $C_d>0$ such that, for every $r\in(0,\frac12)$, we have 
	\begin{equation}\label{mth3e1}
	\prod_{i=1}^3\left(\frac{1}{r^{2+\eps}}\int_{B_r}\frac{|\nabla u_i|^2}{|x|^{d-2}}\,dx\right)\le C_d\left(1+\sum_{i=1}^3\int_{B_1}\frac{|\nabla u_i|^2}{|x|^{d-2}}\,dx\right)^3.
	\end{equation}   
\end{lemma}	

\begin{lemma}[Alt-Caffarelli potential estimate  \cite{alca}]\label{l:potential_estimate}
	For every $u\in H^1(B_r)$ we have the following estimate:
	\begin{equation}\label{e:potential_estimate}
	\frac{1}{r^2}\big|\{u=0\}\cap B_r\big|\left(\mean{\partial B_r}{u\,d\HH^{d-1}}\right)^2\leq C_d\int_{B_r}|\nabla (u-h)|^2\,dx\le C_d\int_{B_r}|\nabla u|^2\,dx,
	\end{equation}
	where:
	\begin{itemize}
		\item $C_d$ is a constant that depends only on the dimension $d$;
		\item $h$ is the harmonic extension of $u$ in $B_r$, that is, 
		$$\Delta h=0\quad\text{in}\quad B_r\ ,\quad u=h\quad\text{on}\quad \partial B_r\,.$$
	\end{itemize}
\end{lemma}

\begin{proof}[\bf Proof of \cref{p:triple}] 
Let $u_+$ and $u_-$ be the first eigenfunctions on $\Omega_+$ and $\Omega_-$, normalized in $L^2(\Omega)$. Let $v\in H^1(\R^d)$ be the $(1+\gamma)$-homogeneous, non-negative harmonic function on $\mathcal C_\delta$, which vanishes on $\partial \mathcal C_\delta$. In polar coordinates 
$$v=r^{1+\gamma}\phi(\theta),$$
where $\phi$ is the first eigenfunction of the spherical Laplacian on $\mathcal C_\delta\cap \mathbb S^{d-1}$, that is,
$$-\Delta_{\mathbb S^{d-1}}\phi=(1+\gamma)(d-1+\gamma)\phi\quad\text{in}\quad \mathcal C_\delta\cap \mathbb S^{d-1},\qquad \phi=0\quad\text{on}\quad \partial \mathcal C_\delta\cap\mathbb S^{d-1},\qquad \int_{\mathbb S^{d-1}}\phi^2(\theta)\,d\theta=1,$$
where we notice that $\gamma$ is uniquely determined by $\delta$ (and the dimension $d$) and 
$$\lim_{\delta\to0}\gamma(\delta)=0.$$
Moreover, we have that 
$$\Delta v\ge 0\quad\text{in sense of distributions in } \R^d.$$
By the three-phase monotonicity formula (\cref{l:three-phase}), that we can apply thanks to~\eqref{e:subharmonicity_better}, there are constants $C>0$ and $\eps>0$ such that 
$$Cr^\eps\ge \left(\frac{1}{|B_r|}\int_{B_r}|\nabla u_+|^2\,dx\right)\left(\frac{1}{|B_r|}\int_{B_r}|\nabla u_-|^2\,dx\right)\left(\frac{1}{|B_r|}\int_{B_r}|\nabla v|^2\,dx\right).$$
Now, using \eqref{e:potential_estimate} and the fact that $\big|\{u_\pm=0\}\cap B_r\big|\ge |\mathcal C_\delta\cap B_r|\ge \frac12|B_r|$, we get 
$$Cr^\eps\ge \left(\mean{\partial B_r}{u_+\,d\HH^{d-1}}\right)^2\left(\mean{\partial B_r}{u_-\,d\HH^{d-1}}\right)^2\left(\frac{1}{|B_r|}\int_{B_r}|\nabla v|^2\,dx\right),$$
for some different constant $C$. Now, using the non-degeneracy (\cref{l:nondegeneracy}), we obtain 
\begin{align*}
Cr^\eps&\ge \frac{1}{|B_r|}\int_{B_r}|\nabla v|^2\,dx\\
&= \frac{1}{|B_r|}\int_0^r\int_{\mathbb S^{d-1}}\Big((1+\gamma)^2\phi^2(\theta)+|\nabla_\theta\phi(\theta)|^2\Big)\rho^{d-1+2\gamma}\,d\theta\,d\rho= (1+\gamma)r^{2\gamma}, 
\end{align*}
which is impossible when $\delta$ (and thus $\gamma$) is small enough ($\eps$ being a fixed constant, depending on $d$, $\lambda_1(\Omega_+)$ and $\lambda_1(\Omega_-)$, but not on $\delta$).
\end{proof}

As a corollary of \cref{p:triple}, we obtain the following regularity result for the solutions of \eqref{e:J_infty}.

\begin{corollary}[Regularity of the one-phase free boundaries]
	Suppose that $D$ is a bounded open set in $\R^d$ with $C^{1,\beta}$ regular boundary, for some $\beta>0$. Suppose that $u\in H^1_0(D)$ is a solution to the problem \eqref{e:J_infty} and that $\Omega_u^+$ and $\Omega 
	_u^-$ are the sets $\Omega_u^+=\{u>0\}$ and $\Omega_u^-=\{u<0\}$. Then: 
	\begin{enumerate}[\rm(i)]
	\item there are no two-phase points on the boundary of $D$, that is, for every $x_0\in \partial D$, there is $\eps>0$ such that 
	$$\Omega_u^+\cap B_\eps(x_0)=\emptyset\quad\text{or}\quad \Omega_u^-\cap B_\eps(x_0)=\emptyset\,;$$
	\item the one-phase free boundaries $\partial\Omega_u^\pm$ are $C^{1,\alpha}$-regular in a neighborhood of $\partial D$. Precisely, if $x_0\in \partial D$ is such that $B_r(x_0)\cap \Omega_u^-=\emptyset$ for some $r>0$, then $\partial\Omega_u^+\cap B_r(x_0)$ is a $C^{1,\alpha}$ manifold, for some $\alpha>0$.
	\end{enumerate}	
\end{corollary}	
\begin{proof}
From \cref{p:equivalence2} point (iv), we know that the sets $\Omega_u^+$ and $\Omega_u^-$ are inwards minimizing. Now, since $\partial D$ is $C^{1,\beta}$ regular, at every point $x_0\in\partial D$ 
%(assume that $x_0=0$) 
there is, up to a rotation of the coordinate system, a cone $\mathcal C_\delta$ contained in $\R^d\setminus D$. Thus, \cref{p:triple} implies that if $x_0\in\partial\Omega_u^+\cap \partial D$, then in a small ball $B_\eps(x_0)$, the set $\Omega_u^-$ is empty. This proves {\rm(i)}. In order to prove {\rm (ii)}, we use again \cref{p:equivalence2} point (v), and we deduce that $\Omega_u^+$ is a solution to the problem 
$$\min\Big\{\lambda_1(\Omega)\ :\ \Omega \text{ quasi-open},\ \Omega\subset \Omega_u^+\cup B_\eps(x_0)\cap D,\ |\Omega|=|\Omega_u^+|\Big\}.$$
Thus, by \cite[Proposition~5.35]{rtv}, $\partial \Omega_u^+$ is $C^{1,\alpha}$ regular in $B_\eps(x_0)$.
\end{proof}

\section{Lipschitz continuous solutions}\label{sect:lip}

In this section, we show that to every solution $\Omega$ of the shape optimization problem \eqref{e:minFLambdaqo}, we can associate a Lipschitz continuous solution $u\in H^1_0(\Omega)$ for the free boundary problem \eqref{e:fb_infty}. Our main result of the section is \cref{l:lipschitz} here below.

\begin{lemma}\label{l:lipschitz}
Let $D$ be a bounded open set in $\R^d$ with $C^{1,\beta}$ regular boundary. Let $\Lambda>0$ be fixed and let $\Omega$ be a solution to \eqref{e:minFLambdaqo}. Then, there exists a function $u:\R^d\to\R$, $u\in H^1_0(\Omega)$, such that:
\begin{itemize}
%\item $u\in H^1_0(\Omega)$; 
\item $u$ is a sign-changing second eigenfunction on $\Omega$;
\item $u$ is a solution to \eqref{e:fb_infty};
\item $u$ is Lipschitz continuous on $\R^d$.
\end{itemize}
%Let $u\in H^1_0(D)$ be a solution to problem~\eqref{eq:mainu}.
%Then it is locally Lipschitz continuous in $D$, after extending it to zero outside $\Omega_u^\pm$.
\end{lemma}
Before dealing with the proof of \cref{l:lipschitz}, we need a technical result. It is well-known that if $\Omega$ minimizes the first eigenvalue among all quasi-open sets with a fixed measure, then the first eigenfunction on $\Omega$ is Lipschitz (when extended as zero outside $\Omega$). This was already proved by Brian\c con and Lamboley in \cite{brla} through an Alt-Caffarelli argument \cite{alca}. In the proof of  \cref{l:lipschitz}, we need to know the Lipschitz constant explicitely, so we briefly give a quantitative local version of this result in the next lemma by a method already used in several other works (see for instance \cite{rtv,bmpv,alca}).

\begin{lemma}\label{l:one-phase_lipschitz_universal}
Suppose that $\Omega$ is a bounded quasi-open set in $\R^d$ and that the function $u\in H^1_0(\Omega)$ is the first eigenfunction on $\Omega$, that is, $u\ge 0$ in $\R^d$, $\ds\int_\Omega u^2\,dx=1$ and  	
$$\lambda_1(\Omega)=\int_\Omega|\nabla u|^2\,dx=\min\left\{\int_\Omega|\nabla \phi|^2\,dx\ :\ \phi\in H^1_0(\Omega),\ \int_\Omega\phi^2\,dx=1\right\}.$$
Suppose that $B_R$ is a ball of radius $R\le 1$ and that there are constants $r>0$ and $K>0$ such that
\begin{equation}\label{e:quasi-minimality}
\int_{\R^d}|\nabla u|^2\,dx\le \frac{\ds\int_{\R^d} |\nabla (u+\varphi)|^2\,dx}{\ds\int_{\R^d} (u+\varphi)^2\,dx}+K\rho^d,
\end{equation}
for every $\varphi\in H^1_0\big(B_\rho(x_0)\big)$ and every ball $B_\rho(x_0)\subset B_R$. Then, there is a constant $C$, depending only on $\lambda_1(\Omega)$, $K$ and $d$, such that, if $u(0)=0$, then  
\[
u(x_0)=0 \qquad \Longrightarrow \qquad \|\nabla u\|_{L^\infty(B_{\sfrac{R}8})}\le C.
\]
\end{lemma}	
\begin{proof}
Let $\varphi\in C^\infty_c(B_\rho(x_0))$ be such that $\varphi=1$ in $B_{\sfrac\rho2}(x_0)$ and $|\nabla \varphi|\lesssim \frac1\rho$. We can compute
\begin{align*}
\lambda_1(\Omega)&\le \frac{\ds\int_{\R^d} |\nabla (u+t\varphi)|^2\,dx}{\ds\int_{\R^d} (u+t\varphi)^2\,dx}+K\rho^d\\
&=\frac{\ds\lambda_1(\Omega)+2t\int \nabla u\cdot\nabla\varphi\,dx+t^2\int|\nabla  \varphi|^2\,dx}{\ds 1+2t\int u\varphi\,dx+t^2\int \varphi^2\,dx}+K\rho^d,
\end{align*}
which implies that 
\begin{align*}
2t\left(-\int \nabla u\cdot\nabla\varphi\,dx+\lambda_1(\Omega)\int u\varphi\,dx\right)\le t^2\int|\nabla  \varphi|^2\,dx+2\left(1+t^2\int \varphi^2\,dx\right)K\rho^d,
\end{align*}
Choosing $t=\rho\le 1$ and using that $\Delta u+\lambda_1(\Omega)u$ is a positive Radon measure on $\R^d$ (see \cref{oss:subharmonicity}), we get  
$$\big(\Delta u+\lambda_1(\Omega)u\big)\big(B_{\sfrac\rho2}(x_0)\big)\le \left(-\int \nabla u\cdot\nabla\varphi\,dx+\lambda_1(\Omega)\int u\varphi\,dx\right)\le C_d(1+K)\rho^{d-1}.$$
As a consequence, if $u(x_0)=0$, using~\cite[formula (2.26)]{bmpv} and an integration by parts we obtain 
\begin{equation}\label{e:mean_boundary_est}
\mean{\partial B_r(x_0)}{u\,d\HH^{d-1}}=\int_0^r\frac{\Delta u(B_\rho(x_0))}{d\,\omega_d\,\rho^{d-1}}\,d\rho\le C_d(1+K)r.
\end{equation}
Now, let $y_0\in B_{\sfrac{R}8}$ and let $x_0$ be the projection of $y_0$ on the set $\{u=0\}$, which is closed as a consequence of \eqref{e:mean_boundary_est}. Since $u(0)=0$, we have that 
$$r_0:=|x_0-y_0|\le \sfrac{R}8.$$
Notice that we have 
$$B_{r_0}(y_0)\subset B_{2r_0}(x_0)\subset B_{\sfrac{R}2}.$$
Thus, applying \eqref{e:mean_boundary_est}, we get  
$$\mean{\partial B_{2r_0}(x_0)}{u\,d\HH^{d-1}}\le C_d(1+K)r_0.$$
Now, since there is a constant $C(d,\lambda_1)$, depending on $d$ and $\lambda_1(\Omega)$ such that 
$$u(x)+C(d,\lambda_1)|x-x_0|^2$$
is subharmonic (see \cref{oss:subharmonicity}), we have that 
\begin{align*}
\|u\|_{L^\infty(B_{{r_0}/2}(y_0))}&\le \mean{B_{r_0}(y_0)}{u\,dx}+C(d,\lambda_1)r_0^2\le 2^d\mean{B_{2r_0}(x_0)}{u\,dx}+C(d,\lambda_1)r_0^2\\
&\le 2^d\left(\mean{\partial B_{2r_0}(x_0)}{u\,d\HH^{d-1}}+C(d,\lambda_1)r_0^2\right)+C(d,\lambda_1)r_0^2\le C_d(1+K)r_0+C(d,\lambda_1)r_0^2.
\end{align*}
Now, the gradient estimate \eqref{e:gradest} gives the claim.
\end{proof}

\begin{proof}[Proof of \cref{l:lipschitz}]
Let $\Omega$ be as in the assumptions. By \cite[Theorem~5.3]{bmpv} there exists a second eigenfunction $w\in H^1_0(\Omega)$, which is Lipschitz continuous on $\R^d$.
We consider two cases.
\medskip

\noindent {\it Case 1.} If $w$ changes sign in $\Omega$, then $w$ is a solution of \eqref{e:fb_infty} (by  \cref{p:equivalence1}), so we can take $u=w$.

\medskip

\noindent {\it Case 2.} Suppose that $w$ does not change sign. Without loss of generality, we can assume that $w$ is non-negative. By \cref{p:equivalence1} and \cref{p:equivalence2}, there is a non-negative first eigenfunction $v\in H^1_0(\Omega)$ such that : 
\begin{itemize}
\item $v$ and $w$ have disjoint supports : $v w=0$ on $\R^d$ ;
\item $w-v$ is a solution to \eqref{e:fb_infty} ;
\item there are positive constants $\alpha$ and $\beta$ such that the function $u:=\alpha w-\beta v$ is a (sign-changing and normalized) second eigenfunction on $\Omega$.
\end{itemize} 
Thus, it remains to prove that $v$ is Lipschitz continuous. Let $x_0\in \Omega_{v}^+:=\{v>0\}$ and let $r$ be the largest radius for which the ball $B_r(x_0)$ is contained in $\{v>0\}$. We fix a constant $r_0>0$ (that we will later choose small enough) and we consider four cases:\medskip

{\it Case 2a.} $r\ge r_0$\,;\medskip

{\it Case 2b.} $r<r_0$ and in $B_{10r}(x_0)$ there is a point lying outside $D$\,;\medskip

{\it Case 2c.}  $r<r_0$, $B_{10r}(x_0)$ is contained in $D$ and in $B_{4r}(x_0)$ there is a point lying in $\Omega_w^+$\,;\medskip

{\it Case 2d.}  $r<r_0$, $B_{10r}(x_0)$ is contained in $D$ and $B_{4r}(x_0)\cap\{w>0\}=\emptyset$. \medskip

We start with the case $2a$. Since $v$ solves 
$$-\Delta v=\lambda_2(\Omega)v\quad\text{in}\quad B_r(x_0),$$
the classical gradient estimate (see~\cite{gitr}) gives
\begin{equation}\label{e:gradest}
\|\nabla v\|_{L^\infty(B_{\sfrac{r}2}(x_0))}\le C_d\|\lambda_2(\Omega)v\|_{L^\infty(B_{r}(x_0))}+\frac{2d}{r}\|v\|_{L^{\infty}(B_r(x_0))}.
\end{equation}
Since $v$ satisfies the global $L^\infty$ bound $\|v\|_{L^\infty(\R^d)}\le C_d\big(\lambda_2(\Omega)\big)^{\sfrac{d}4}$ and since $r\ge r_0$, we get that there is a constant $C(d,\lambda_2,r_0)$, depending on $d$, $\lambda_2(\Omega)$ and $r_0$, such that
$$|\nabla v|(x_0)\le C(d,\lambda_2,r_0).$$

We now consider the case $2b$. Let $w_D$ be the solution of 
$$-\Delta w_D=1\quad\text{in}\quad D,\qquad w_D=0\quad\text{on}\quad\R^d\setminus D.$$
Since $D$ is $C^{1,\beta}$ regular, the function $w_D$ is Lipschitz continuous on $\R^d$. We denote by $L$ its Lipschitz constant. Setting 
$$\ds C:=C_d\big(\lambda_1(\Omega)\big)^{\sfrac{(d+4)}4}$$ 
to be the constant from \eqref{e:subharmonicity_better}, we know that $Cw_D\ge v$ everywhere in $\R^d$. Then, we have 
$$v\le 11CLr\quad\text{in}\quad B_r(x_0).$$
Using again the gradient estimate \eqref{e:gradest}, we get that there is a constant $C(D,d,\lambda_2)$ depending only on $D$, $d$ and $\lambda_2(\Omega)$ such that 
$$|\nabla v|(x_0)\le C(D,d,\lambda_2).$$

We next consider the case $2c$. Let $y_0$ be a point in $\{w>0\}$. By the two-phase monotonicity formula of Caffarelli-Jerison-K\"enig (see \cite{CJK} and \cite{mono}), we know that there is a constant $C$, depending on $\lambda_2(\Omega)$ and the dimension such that 
$$C\ge \left(\mean{B_R(y_0)}{|\nabla w|^2\,dx}\right)\left(\mean{B_R(y_0)}{|\nabla v|^2\,dx}\right).$$
Applying \cref{l:potential_estimate}, we get that (up to multiplying $C$ by a factor depending only on the dimension)
$$C\ge \frac{|\{w=0\}\cap B_R(y_0)|}{|B_R|}\left(\frac1R\mean{\partial B_R(y_0)}{w\,d\HH^{d-1}}\right)^2\frac{|\{v=0\}\cap B_R(y_0)|}{|B_R|}\left(\frac1R\mean{\partial B_R(y_0)}{v\,d\HH^{d-1}}\right)^2.$$
and, since $w$ and $v$ have disjoint supports, 
$$C\ge \frac{|\{v>0\}\cap B_R(y_0)|}{|B_R|}\left(\frac1R\mean{\partial B_R(y_0)}{w\,d\HH^{d-1}}\right)^2\frac{|\{w>0\}\cap B_R(y_0)|}{|B_R|}\left(\frac1R\mean{\partial B_R(y_0)}{v\,d\HH^{d-1}}\right)^2.$$
We next choose $R=4r$. Thus, the non-degeneracy of $w$ (in order to use the non-degeneracy  \cref{l:nondegeneracy}, we choose $r_0$ small enough from the beginning) gives that 
$$\frac1R\mean{\partial B_R(y_0)}{w\,d\HH^{d-1}}\ge \eta.$$
In particular, there is a point $z_0\in\partial B_R(y_0)$ such that $w(z_0)\ge \eta$. But now, the Lipschitz continuity of $w$ (say $|\nabla w|\le L_w$) gives that $w>0$ in $B_{\eta/L_w}(z_0)$. Thus, we get also that 
$$\frac{|\{w>0\}\cap B_R(y_0)|}{|B_R|}\ge \frac{|B_{\eta/L_w}(z_0)\cap B_R(y_0)|}{|B_R|}\ge C(L_w,\eta,d).$$
Similarly, since $v$ is positive in $B_r(x_0)\cap B_{4r}(y_0)$, we have that there is a dimensional constant $c_d$ such that 
$$\frac{|\{v>0\}\cap B_R(y_0)|}{|B_R|}\ge c_d.$$
This finally gives that there is a constant $C(w,d)$, depending on $w$ and the dimension, such that 
$$C\ge \frac1R\mean{\partial B_R(y_0)}{v\,d\HH^{d-1}}.$$
Applying \eqref{e:L2-Linfty} the gradient estimate as in the case {\it 2a}, we get that 
$$|\nabla v|(x_0)\le C(w,d).$$
Finally, we consider the case {\it 2d}. First of all, we suppose that there is at least one point $x_1\in\partial\Omega_v\cap D$ and a radius $r_1>0$ such that $B_{r_1}(x_1)\subset D$ and $B_{r_1}(x_1)\cap\{w>0\}=\emptyset$ (in fact if there were not such $x_1$ and $r_1$, the proof of the lemma would be concluded with the case {\it 2c}). Now, by \cite{brla}, we know that $v$ is Lipschitz in $B_{r_1}(x_1)$ and that the free boundary $\partial\Omega_v\cap  B_{r_1}(x_1)$ is $C^\infty$ up to a small closed set. In particular, we may assume that in $D$ there are two distinct points $x_1$ and $x_2$, and a radius $0<R_{12}<\frac13|x_1-x_2|$ such that:
\begin{itemize}
\item $B_{R_{12}}(x_1)\subset D$ and $B_{R_{12}}(x_2)\subset D$\,;
\item $B_{R_{12}}(x_1)\cap \{w>0\}=\emptyset$  and $B_{R_{12}}(x_2)\cap \{w>0\}=\emptyset$\,;
\item $\partial\Omega_v$ is $C^\infty$ in $B_{R_{12}}(x_1)$ and $B_{R_{12}}(x_2)$\,;
\item there are constants $m>0$ and $C>0$ such that, for every $i=1,2$ and every $t\in(-m,m)$, there is a function $v_{i,t}\in H^1(B_{R_{12}}(x_i))$ such that:
$$v_{i,t}=v\quad\text{on}\quad \partial B_{R_{12}}(x_i).$$ 
\begin{equation}\label{e:vector_field1}
\big|B_{R_{12}}(x_1)\cap \{v_{i,t}>0\}\big|-\big|B_{R_{12}}(x_1)\cap \{v>0\}\big|=t
\end{equation}
\begin{equation}\label{e:vector_field2}
\int_{B_{R_{12}}(x_1)}\big(|\nabla v_{i,t}|^2+v_{i,t}^2\big)\,dx\le Kt.
\end{equation}
\end{itemize}	
We notice that for the construction of $v_{i,t}$ it is sufficient to take smooth vector fields $\xi_i\in C^\infty_c\big(B_{R_{12}}(x_i);\R^d\big)$, $i=1,2$, orthogonal to $\partial\Omega_v$ (parallel to the outgoing normal $\nu$) and pointing outwards and to define the functions 
$$v_{i,t}(x):=v(x+\xi_{i,t}(x)).$$
the claims \eqref{e:vector_field1} and \eqref{e:vector_field2} now follow from the well-known (see \cite{brla}) first variation formulas 
$$\frac{d}{dt}\Big|_{t=0}\int|\nabla v_{i,t}|^2\,dx=-\int_{\partial\Omega_v}(\xi\cdot \nu)|\nabla v|^2\,d\HH^{d-1}\qquad\text{and}\qquad \frac{d}{dt}\Big|_{t=0}|\{v_{i,t}>0\}|=\int_{\partial\Omega_v}\xi\cdot \nu\,d\HH^{d-1},$$
and the inverse function theorem. Now, with this family of functions in hand, we get back to the case {\it 2d}. Notice that, by choosing $r_0>0$ small enough, we can assume that the ball $B_{4r}(x_0)$ intersects at most one of the balls $B_{R_{12}}(x_1)$ and $B_{R_{12}}(x_2)$ (say, the first one). Thus, if $\varphi$ is a function compactly supported in $B_{4r}(x_0)$, we can consider the competitor 
$$\widetilde v=\begin{cases}
v\quad&\text{in}\quad \R^d\setminus (B_{4r}(x_0)\cup B_{R_{12}}(x_1)),\\
v+\varphi\quad&\text{in}\quad B_{4r}(x_0),\\
v_{i,t}\quad&\text{in}\quad B_{R_{12}}(x_1),\\
\end{cases}$$
where we choose $t$ such that $\{\widetilde v>0\}=\{v>0\}$. Thus, from \eqref{e:vector_field1} and \eqref{e:vector_field2}, we get that $v$ satisfies the almost-minimality condition \eqref{e:quasi-minimality}. Thus, we can use the universal estimate from \cref{l:one-phase_lipschitz_universal} and this concludes the proof. 
\end{proof}

%	
%\newpage

\section{First variation formula}\label{sect:approx}
Let $\Omega$ be a solution to \eqref{e:minFLambdaqo}. 
From now on, we will take $\Lambda=1$, without loss of generality.
We know that there is a sign-changing function $u\in H^1_0(\Omega)$, which is Lipschitz continuous on $\R^d$ and a solution to \eqref{e:fb_infty}.  Our next objective is to prove that the function $u$ is a solution, in the viscosity sense, of a free boundary problem. In order to do so, we will first try to deduce a first order optimality condition coming from internal perturbations with vector fields. Since the function $\R^2\ni (a,b)\mapsto \max\{a,b\}$ is not differentiable, we will approximate $J_\infty$ with smooth functionals, inspired by~\cite{rtt} and \cite{krli}.

In what follows we will use the notation 
$$\mathcal R(v):=\frac{\ds\int_{\{v>0\}}|\nabla v|^2\,dx}{\ds\int_{\{v>0\}} v^2\,dx}\quad\text{for every nonnegative function}\quad v\in H^1(\R^d),\ v\neq0,$$
while, when $v=0$, we simply set $\mathcal R(0)=+\infty$. For every $p\in(1,+\infty)$, we consider the problem
\begin{equation}\label{eq:mainp}
\min\left\{J_p\big(\mathcal R(v_+)\,;\,\mathcal R(v_-)\big)+\int_{D}|u-v|^2+|\Omega_v| : v\in H^1_0(D)\right\}.
\end{equation}
where as usual $v_+=\max\{v,0\}$, $v_-=\max\{-v,0\}$ and $\Omega_v=\{v\neq0\}$, and where $J_p$ is the function
$$J_p(X,Y):=\big(X^p+Y^p\big)^{\sfrac1p}.$$
\begin{remark}
For all $p\in(1,+\infty)$, there exists a solution to the problem~\eqref{eq:mainp}: the proof is standard and follows by the same argument as the one in \cref{s:existenceqo}.
\end{remark}

\begin{lemma}[Convergence of the minima]\label{l:convergence}
For every $p\ge 2$, let $v_p\in H^1_0(D)$ be a solution to \eqref{eq:mainp} such that
$$\int_D(v_p^+)^2\,dx=\int_D(v_p^-)^2\,dx=1.$$
Then, as $p\to\infty$, $v_p$ converges strongly in $H^1_0(D)$ to the function $u$, solution to~\eqref{e:fb_infty}. Moreover, the characteristic functions $\ind_{\Omega_{v_p}^+}$ and $\ind_{\Omega_{v_p}^-}$ converge strongly in $L^1$ and pointwise almost-everywhere to $\ind_{\Omega_{u}^+}$ and $\ind_{\Omega_{u}^-}$, respectively.
\end{lemma}
\begin{proof}
We first notice that, by testing the minimality of $v_p$ with $u$, we get 
\begin{align*}
J_p\left(\int_{D}|\nabla v_p^+|^2\ ;\,\int_{D}|\nabla v_p^-|^2\right)+\int_{D}|u-v_p|^2+|\Omega_{v_p}|&\le J_p\left(\int_{D}|\nabla u_+|^2\ ;\,\int_{D}|\nabla u_-|^2\right)+|\Omega_{u}|\\
&\le 2\,J_\infty\left(\int_{D}|\nabla u_+|^2\ ;\,\int_{D}|\nabla u_-|^2\right)+|\Omega_{u}|.
\end{align*}
Thus, $v_p$ is bounded in $H^1$ and so, up to a subsequence, $v_p^+$ and $v_p^-$ converge weakly in $H^1$, strongly in $L^2$ and pointwise almost-everywhere to a function $v_\infty\in H^1_0(D)$. The convergence and the minimality of $v_p$ now give 
\begin{align*}
J_\infty\left(\int_{D}|\nabla u_+|^2\ ;\,\int_{D}|\nabla u_-|^2\right)+&\int_{D}|u-v_\infty|^2+|\Omega_{u}|\\
&\le J_\infty\left(\int_{D}|\nabla v_\infty^+|^2\ ;\,\int_{D}|\nabla v_\infty^-|^2\right)+\int_{D}|u-v_\infty|^2+|\Omega_{v_\infty}|\\
&\le 
J_\infty\left(\liminf_{p\to\infty}\int_{D}|\nabla v_p^+|^2\ ;\,\liminf_{p\to\infty}\int_{D}|\nabla v_p^-|^2\right)+\int_{D}|u-v_\infty|^2+|\Omega_{v_\infty}|\\
&\le \liminf_{p\to\infty}\left\{
J_p\left(\int_{D}|\nabla v_p^+|^2\ ;\,\int_{D}|\nabla v_p^-|^2\right)+\int_{D}|u-v_p|^2+|\Omega_{v_p}|\right\}\\
&\le  \lim_{p\to\infty} \left\{J_p\left(\int_{D}|\nabla u_+|^2\ ;\,\int_{D}|\nabla u_-|^2\right)+|\Omega_{u}|\right\}\\
&= J_\infty\left(\int_{D}|\nabla u_+|^2\ ;\,\int_{D}|\nabla u_-|^2\right)+|\Omega_{u}|,
\end{align*}
which proves that $v_\infty=u$ and that all the inequalities above are equalities. In particular,
\begin{align*}
\int_{D}|\nabla u_+|^2\le \liminf_{p\to\infty}\int_{D}|\nabla v_p^+|^2
&= J_\infty\left(\liminf_{p\to\infty}\int_{D}|\nabla v_p^+|^2\ ;\,\liminf_{p\to\infty}\int_{D}|\nabla v_p^-|^2\right)\\
&= J_\infty\left(\int_{D}|\nabla u_+|^2\ ;\,\int_{D}|\nabla u_-|^2\right)=\int_{D}|\nabla u_+|^2,
\end{align*}
which means that the convegence is strong in $H^1$. Finally, the strong convergence of the characteristic functions follows from the equalities 
\[
|\Omega_u^\pm|=\liminf_{p\to\infty}|\Omega_{v_p}^\pm|.\qedhere
\]
\end{proof}

In what follows we will use the notation $\delta J(u)[\xi]$ to indicate  the first variation of a functional $J$ at a function $u$ in the direction of a smooth vector field $\xi$. Precisely, for every $u\in H^1_0(D)$, $\xi\in C^\infty_c(D;\R^d)$, we define the diffeomorphism $\Phi_t$ as 
\begin{equation}\label{e:diffeo}
\Phi_t=\Psi_t^{-1}, \quad\text{where}\quad \Psi_t(x):=x+t\xi(x).
\end{equation}
Then, if the derivative $\ds \frac{\partial}{\partial t}\Big|_{t=0} J(u\circ\Phi_t)$ exists, we set 
$$\delta J(u)[\xi]:=\frac{\partial}{\partial t}\Big|_{t=0} J(u\circ\Phi_t).$$

It is well-known that $\delta \mathcal R(u)[\xi]$ exists for any $u\in H^1_0(D)$ and $\xi\in C^\infty_c(D;\R^d)$ and that  
\begin{equation}\label{e:first_variation_J}
\delta \mathcal R(u)[\xi]=\int_{D}\Big(|\nabla u|^2-\lambda u^2\Big)\mathrm{div}\,\xi-2\,\nabla u\, D\xi\,(\nabla u)^t\,dx,
\end{equation}
where $\lambda:=\mathcal R(u)$.
 Moreover, setting $\text{\rm Vol}(u)=|\Omega_u|$, we have that $\delta\text{\rm Vol}(u)[\xi]$  exists for all $u\in H^1_0(D)$ and $\xi\in C^\infty_c(D;\R^d)$, and  
\begin{equation}\label{e:first_variation_Vol}
\delta\text{\rm Vol}(u)[\xi]=\int_{\Omega_u}\text{\rm div}\,\xi\,dx.
\end{equation}
Now, using the formulas \eqref{e:first_variation_J} and \eqref{e:first_variation_Vol}, we can compute the optimality condition for the minimizers of \eqref{eq:mainp}. Precisely, we have the following lemma. 

\begin{lemma}\label{le:optp}
	Let $p>1$ and let $u_p\in H^1_0(D)$ be a solution to~\eqref{eq:mainp} such that 
	$$\int_D(u_p^+)^2\,dx=\int_D(u_p^-)^2\,dx=1.$$
	Then, setting 
	\[
	a^\pm_p:=\frac{\big(\mathcal R(u_p^\pm)\big)^{p-1}}{\Big[\big(\mathcal R(u_p^+)\big)^p+\big(\mathcal R(u_p^-)\big)^p\Big]^{1-\frac1p}},
	\]
	we have that for any smooth vector field $\xi\in C^\infty_c(D;\R^d)$,
\begin{equation}\label{e:first_variation_full}
a^+_p\delta \mathcal R(u_p^+)[\xi]+a_p^-\delta\mathcal R(u_p^-)[\xi]+2\int_{D}(u_p-u)\,\xi\cdot \nabla u_p\,dx+\delta \text{\rm Vol}(u_p^+)[\xi]+\delta \text{\rm Vol}(u_p^-)[\xi]=0.
\end{equation}
\end{lemma}
\begin{proof}
Let $\xi\in C^\infty_c(D;\R^d)$ and $\Phi_t$ be as in \eqref{e:diffeo}. Since we already have \eqref{e:first_variation_J} and \eqref{e:first_variation_Vol}, it is sufficient to compute the variation of the fidelity term. We have 
$$\frac{\partial}{\partial t}\Big|_{t=0}\int_D|u_p\circ\Phi_t-u|^2\,dx=	\int_{D}2(u_p-u)\,\xi\cdot \nabla u_p\,dx.$$
Then, using the optimality of $u_p$, we get that 
\[
	\begin{split}
	0&=\frac{\partial}{\partial t}\Big|_{t=0}\Big[ J_p\Big(\mathcal R\big((u_p\circ\Phi_t)_+\big), \mathcal R\big((u_p\circ\Phi_t)_-\big)\Big)+\int_D|u_p\circ\Phi_t-u|^2\,dx+\big|\{u_p\circ\Phi_t\neq0\}\big|\Big]\\
	&= a_p^+\delta \mathcal R(u_p^+)[\xi]+a_p^-\delta \mathcal  R(u_p^-)[\xi]+\int_{D}2(u_p-u)\,\xi\cdot \nabla u_p\,dx+\delta \text{\rm Vol}(u_p^+)[\xi]+\delta \text{\rm Vol}(u_p^-)[\xi],
	\end{split}
	\]
	which gives the claim.
\end{proof}

We now pass to the limit as $p\rightarrow +\infty$.

\begin{lemma}\label{l:limit_optimality}	
Let $D$ be a bounded open set and let $u\in H^1_0(D)$ be a Lipschitz continuous solution of \eqref{e:fb_infty}.	
Then, there are constants $a_+\ge 0$ and $a_-\ge 0$ such that
$$a_++a_-=1,$$ 
and, for every smooth vector field $\xi\in C^\infty_c(D;\R^d)$, we have 
\begin{align*}
a_+\delta \mathcal R(u_+)[\xi]+a_-\delta \mathcal R(u_-)[\xi]+\delta \text{\rm Vol}(u_+)[\xi]+\delta \text{\rm Vol}(u_-)[\xi]=0.
\end{align*}	
\end{lemma}
\begin{proof}
Using \cref{le:optp} and the convergence of the solutions $u_p$ proved in \cref{l:convergence}, we have that the variation of the fidelity term vanishes. Indeed, 
$$\lim_{p\to\infty}\int_{D}(u_p-u)\,\xi\cdot \nabla u_p\,dx=0.$$
Thus, passing to the limit in \cref{le:optp} and using again \cref{l:convergence}, we get the claim. Finally, the equality $a_++a_-=1$ follows from the fact that 
\[
\lim_{p\to\infty}a_p^\pm=a_\pm\qquad\text{and}\qquad (a_p^+)^{\frac{p}{p-1}}+(a_p^-)^{\frac{p}{p-1}}=1\quad\text{for every}\quad p\ge 1.\qedhere
\]
\end{proof}

\section{Two-phase free boundary: blow-up limits and regularity}\label{sect:blowup}
Let $u:\R^d\to\R$ be a Lipschitz continuous solution to~\eqref{e:fb_infty}.
Let $x_0$ be a point of the free boundary, that is, \[
x_0\in \partial \Omega_u\cap D,
\]
and we define the rescaled function \[
u_{x_0,r}(x)=\frac{u(x_0+rx)}{r},\qquad \text{for }r>0,
\]
on the set $\{x\in \R^d : x_0+rx\in D\}$.
For any vanishing sequence $(r_n)$, we say that $u_{x_0,r_n}$ is a blow-up sequence (with fixed center).
It is clear that, for all $R>0$, for all $n$ large enough, we have \[
B_R\subset \{x\in \R^d : x_0+r_nx\in D\},
\]
and moreover, by Lipschitz continuity of $u$ and the definition of the blow-up sequence with $u(x_0)=0$, we have that there is a locally Lipschitz continuous function $u_0\colon \R^d\rightarrow \R$ such that
\begin{equation}\label{eq:blowup}
\|u_{x_0,r_n}-u_0\|_{L^\infty(B_R)}\rightarrow 0,\qquad \text{for all }R>0,
\end{equation}
up to pass to a suitable subsequence with a diagonal argument.
\begin{definition}
We will say that $u_0\colon \R^d\rightarrow \R$ is a \emph{blow-up} limit of $u$ at $x_0$. 
%The set of all blow-up limits of $u$ at $x_0$ is denoted by $\mathcal B\mathcal U_u(x_0)$.
\end{definition}

Our main result is the following 

\begin{theorem}\label{t:main_intext}
Let $u$ be a Lipschitz continuous solution of \eqref{e:fb_infty} and let $a_+\ge 0$ and $a_-\ge 0$ be the constants from \cref{l:limit_optimality}. Then 
\begin{equation}\label{e:non-degeneracy-a}
a_+>0\qquad\text{and}\qquad a_->0.
\end{equation}
Moreover, if $x_0\in\partial\Omega_u^+\cap\partial\Omega_u^-$ then every blow-up limit $u_0$ of $u$ at $x_0$ is of the form 
\begin{equation}\label{e:classification_blowup_limits}
u_0(x):=\beta_+(x\cdot\nu)_+-\beta_-(x\cdot\nu)_-,
\end{equation}
where $\nu\in \partial B_1$ and the coefficients $\beta_+$ and $\beta_-$ are such that 
\begin{equation}\label{e:classification_blowup_limits2}
\beta_+\ge \frac1{\sqrt a_+},\quad \beta_-\ge\frac1{\sqrt a_-}\quad\text{and}\quad a_+\beta_+^2=a_-\beta_-^2.
\end{equation}
\end{theorem}	
	
As a corollary, we obtain the regularity of the two-phase free boundary. 

\begin{corollary}\label{c:main_cor}
Let $D$ be a bounded open set and let $u:D\to\R$ be a Lipschitz continuous solution to \eqref{e:fb_infty}. Then, in a neighborhood of the two-phase free boundary $\partial\Omega_u^+\cap\partial\Omega_u^-$, both $\partial\Omega_u^+$ and $\partial\Omega_u^-$ are $C^{1,\alpha}$ regular.
\end{corollary}	 
\begin{proof}
We define the function $v$ as 
$$v=\sqrt a_+\,u_+-\sqrt a_-\,u_-.$$
Then:
\begin{itemize}
\item $v$ is Lipschitz continuous;
\item $v$ satisfies the equations 
$$-\Delta v=\lambda v\quad\text{in}\quad \Omega_v^+\cup\Omega_v^-.$$
\item on the one-phase free boundaries $D\cap\partial\Omega_v^+\setminus\partial\Omega_v^-$ and  $D\cap\partial\Omega_v^-\setminus\partial\Omega_v^+$, we have that  $|\nabla v|=1$ in viscosity sense (see for example \cite[Section~5]{rtv});
\item for every two-phase point $x_0\in\partial\Omega_v^+\cap\partial\Omega_v^-$, $v$ satisfies the equations
$$|\nabla v_+|\ge 1,\quad |\nabla v_-|\ge 1,\quad |\nabla v_+|=|\nabla v_-|,$$
in viscosity sense. This is an immediate consequence of the classification of the blow-up limits of  \cref{t:main_intext}, and can be done as in \cite[Section 2]{despve}.
\end{itemize}
Thus, the claim follows from \cite[Theorem~1.1 and 4.3]{despve}.
\end{proof}	

\subsection{Convergence of the blow-up sequences}
In this section we prove the strong convergence of the blow-up sequences. The main result is the following.

\begin{lemma}\label{le:convH1bu}
	Let $D$ be an open subset of $\R^d$, $u$ a Lipschitz continuous solution of \eqref{e:fb_infty} and $y_0\in \partial \Omega_u\cap D$.
	Let $r_n>0$ be a vanishing sequence and $u_n:=u_{y_0,r_n}$ be the corresponding blow-up sequence converging locally uniformly to the blow-up limit $u_0\colon \R^d\rightarrow \R$. Then, for every $R > 0$, 
	\begin{enumerate}[\rm (i)]
	\item the sequence of rescalings $u_{y_0,r_n}$ converges strongly in $H^1(B_R)$ to $u_0$;
	\item the sequences of characteristic functions $\ind_{\Omega_n^+}$ and $\ind_{\Omega_n^-}$, where $\Omega_n^\pm:=\{\pm u_{n}>0\}$, converge in $L^1(B_R)$ and pointwise almost-everywhere to the characteristic functions $\ind_{\Omega_0^+}$ and $\ind_{\Omega_0^-}$ of the sets $\Omega_0^\pm:=\{\pm u_0>0\}$.
	\end{enumerate}
\end{lemma}
\begin{proof}
	We first prove {\rm (i)}. We will proceed as in \cite[Step 5 of the proof of Theorem~3.1]{rtt}. We notice that $u_n$ is a weak (in $H^1(\R^d)$) solution of the equation
	\begin{equation}\label{e:roma1}
	\Delta u_n^\pm+r_n\lambda_1(\Omega^\pm)u_n^\pm=\mu_n^\pm\quad\text{in}\quad \R^d,
	\end{equation}
	for certain positive Radon measures $\mu_n^+$ and $\mu_n^-$.
	On the other hand $u_0^+$ and $u_0^-$ are nonnegative and harmonic on $\{u_0>0\}$ and $\{u_0<0\}$. Thus, there are positive Radon measures $\mu^+$ and $\mu^-$ such that 
			\begin{equation}\label{e:roma2}
			\Delta u_0^\pm=\mu^\pm\quad\text{in}\quad \R^d.
				\end{equation}
%	Let $\varphi\in H^1(\R^d)$ be fixed. Then, passing to the limit as $n\to\infty$, we get  
%	$$\int_{\R^d}\varphi\,d\mu_n^\pm=-\int_{\R^d}\nabla\varphi\cdot\nabla u_n^\pm\,dx+r_n\lambda_1(\Omega^\pm)\int_{\R^d} \varphi u_n^\pm\,dx\to -\int_{\R^d}\nabla\varphi\cdot\nabla u_0^\pm\,dx=\int_{\R^d}\varphi\,d\mu^\pm.$$
Let now $R>0$ be fixed. Since $u_n$ and $u_0$ are uniformly Lipschitz continuous in $B_R$, there is a constant $C_R>0$, depending only on $R$ such that 
$$\mu^\pm_n(B_R)+\mu^\pm(B_R)\leq C_R\qquad\text{for every}\qquad n\ge 0.$$
Let now $\varphi\in  C^\infty_0(\R^d)$, be a test function such that  
$$0\le \varphi\le 1\quad\text{in}\quad \R^d\ ,\quad \varphi=1\quad\text{in}\quad B_R\ ,\quad\text{and}\quad \varphi=0\quad\text{in}\quad \R^d\setminus B_{2R}.$$
	We test the difference of the two equations \eqref{e:roma1} and \eqref{e:roma2} with $\varphi(u_n^\pm-u^\pm)$
	\[
	\int_{\R^d}\nabla (u^\pm_n-u_0^\pm)\cdot\nabla [\varphi(u^\pm_n-u_0^\pm)]\,dx=\int_{\R^d}\varphi(u^\pm_n-u_0^\pm)\,d(\mu^\pm-\mu^\pm_n)+r_n\lambda_1(\Omega^\pm)\int_{\R^d}	\varphi\,u^\pm_n\,(u_n^\pm-u^\pm_0)\,dx.
	\]
	We now observe that, first of all, by definition of $\varphi$
	\begin{align*}
	\int_{B_R}|\nabla (u^\pm_n-u_0^\pm)|^2\,dx&\le \int_{B_{2R}}\varphi|\nabla (u^\pm_n-u_0^\pm)|^2\,dx\\
	&\le \int_{\R^d}\nabla (u^\pm_n-u_0^\pm)\cdot\nabla [\varphi(u^\pm_n-u_0^\pm)]\,dx-\int_{B_{2R}}(u^\pm_n-u^\pm_0)\nabla (u^\pm_n-u^\pm_0)\cdot \nabla \varphi\,dx.
	\end{align*}
	It is easy to check that, thanks to the weak $H^1_{loc}$ convergence and the uniform convergence  
\[
\lim_{n\to\infty}\|u^\pm_n-u^\pm\|_{L^\infty(B_{2R})}= 0,
\]
therefore we get that the last term in the right-hand side converges to zero as $n\to\infty$. Moreover, we have
$$\lim_{n\to\infty}\lambda_1(\Omega^\pm)\int_{\R^d}\varphi\, u^\pm_n(u_n^\pm-u^\pm_0)\,dx=0.$$	
	Finally, using again the local uniform convergence, we get 
	\begin{align*}\left|\int_{\R^d}\varphi(u^\pm_n-u_0^\pm)\,d(\mu^\pm_n-\mu^\pm)\right|&\leq \Big(\mu_n^\pm(B_{2R})+\mu_0^\pm(B_{2R})\Big)\|u^\pm_n-u_0^\pm\|_{L^\infty(B_{2R})}\\
&\leq C_{2R}\|u^\pm_n-u_0^\pm\|_{L^\infty(B_{2R})}	\rightarrow 0,	\end{align*}
which finally implies that $u^\pm_n$ strongly converges to $u_0^\pm$ in $H^1(B_R)$.
\smallskip

We now prove {\rm (ii)}. We will show that $\ind_{\Omega_n^+}$ converges pointwise almost-everywhere to $\ind_{\Omega_0^+}$. We first consider the case when $x_0\in\R^d$ is a point of Lebesgue density one for $\Omega_0^+$. If $x_0\in \Omega_0^+$, then $u_0(x_0)>0$ and by the uniform convergence of $u_{n}$ to $u_0$, we get that $u_n(x_0)>0$ for $n$ large enough. This gives that 
$$\ind_{\Omega_0^+}(x_0)=1=\lim_{n\to\infty}\ind_{\Omega_n^+}(x_0).$$
We will next show that $x_0$ cannot be on the boundary of $\Omega_0^+$. Let $\rho>0$ be fixed and small. If there was a sequence of points $x_n$ converging to $x_0$ such that $u_n(x_n)<0$, then by the nondegeneracy of $u_n^-$ we have that $\|u_n^-\|_{L^\infty(B_{\rho}(x_n))}>\rho\eta$, which passing to the limit as $n\to 0$ implies that $\|u_0^-\|_{L^\infty(B_{2\rho}(x_0))}>\rho\eta$. Thus, the $L$-Lipschitz continuity of $u_0^-$ implies that in $B_{3\rho}(x_0)$ there is a ball of radius $\rho\eta/L$, where $u_0^-$ is strictly positive (and so $u_0^+$ is zero). Since $\rho$ is arbitrary, we obtain a contradiction with the fact that $x_0$ is of density $1$ for $\Omega_0^+$. This means that there is a ball $B_{r_0}(x_0)$ such that $\Omega_n^-\cap B_{r_0}(x_0)=\emptyset$, for every $n$ large enough. In particular, in $B_{r_0}(x_0)$ the function $u_0^+$ is a blow-up limit of eigenfunctions on optimal sets for the first eigenvalue $\lambda_1$. Thus, by \cite{rtv}, $u_0$ is a local minimizer of the one-phase Alt-Caffarelli functional and so, it satisfies an exterior density estimate, that is, there are no points of density one on the boundary of $\Omega_0^+$. This concludes the proof in the case when $x_0$ has density one. 

Let now $x_0$ be a point of Lebesgue density $0$ for $\Omega_0^+$. By the continuity of $u_0^+$ we have that $u_0^+(x_0)=0$ and $\ind_{\Omega_u^+}(x_0)=0$. Suppose for the sake of contradiction that (for some subsequence that we still denote by $u_n^+$) $u_n^+(x_0)>0$ for every $n>0$. But then the nondegeneracy of $u_n$ at $x_0$ implies that there is a constant $\eta>0$ such that
$$\|u_n^+\|_{L^\infty(B_\rho(x_0))}>\eta\rho,$$
for every $\rho>0$ and every $n\ge0$. As a consequence, the uniform $L$-Lipschitz continuity of $u_n$ implies that there are points $x_n\in B _\rho(x_0)$ such that 
$$u_n^+\ge \frac\eta2\quad\text{in}\quad B_{\rho\eta/2L}(x_n).$$
Notice that, up to extracting a subsequence $x_n$ converges to some point $x_\infty\in \overline B_\rho(x_0)$. The uniform convergence of $u_n^+$ now implies that
$$u_0^+\ge \frac\eta2\quad\text{in}\quad B_{\rho\eta/2L}(x_\infty).$$
Since $\rho$ is arbitrary this contradicts the initial assumption that $x_0$ has Lebesgue density $0$. Thus, we get that for $n$ large enough 
$u_n^+(x_0)=0$, which implies that 
$$\ind_{\Omega_0^+}(x_0)=0=\lim_{n\to\infty}\ind_{\Omega_n^+}(x_0),$$
and this concludes the proof.
\end{proof}

As an immediate corollary of \cref{le:convH1bu} and \cref{l:limit_optimality}, we obtain the following stationarity condition for the blow-up limits of $u$.
\begin{lemma}
Let $u$ be a Lipschitz continuous solution of ~\eqref{e:fb_infty} in the open set $D\subset\R^d$ and let $x_0\in\partial\Omega_u\cap D$. Then, for every  blow-up limit $u_0:\R^d\to\R$ of $u$ at $x_0$, we have the first variation formula
\begin{align}
0&=a_+\int_{\R^d}|\nabla u_0^+|^2\mathrm{div}\,\xi-2\,\nabla u_0^+\, D\xi\,(\nabla u_0^+)^t\,dx\notag\\
&\qquad+a_-\int_{\R^d}|\nabla u_0^-|^2\mathrm{div}\,\xi-2\,\nabla u_0^-\, D\xi\,(\nabla u_0^-)^t\,dx+\int_{\Omega_{u_0}}\text{\rm div}\,\xi\,dx\,,\label{e:first_variation_u0}
\end{align}
for every smooth vector field $\xi\in C^\infty_c(\R^d;\R^d)$, where $a_+$ and $a_-$ are the nonnegative constants from \cref{l:limit_optimality}.
\end{lemma}	

\subsection{Homogeneity of the blow-up limits}\label{sect:weiss}
For every $u\in H^1(B_1)$, we consider the following Weiss-type boundary adjusted energy  
\begin{equation}\label{eq:weiss}
\begin{split}
W(u)&=\left[a^+\int_{B_1}|\nabla u_+|^2\,dx+a^-\int_{B_1}|\nabla u_-|^2\,dx\right]\\
&\qquad -\left[a^+\int_{\partial B_1}u_+^2\,d\mathcal H^{d-1}+a^-\int_{\partial B_1}u_-^2\,d\mathcal H^{d-1}\right]+|\Omega_u^+\cup\Omega_u^-\cap B_1|.
\end{split}
\end{equation}
We will prove a monotonicity formula for $W$, which we will use to show that the blow-up limits are $1$-homogeneous functions. The argument is standard (see \cite{weiss}) and is based on the first variation formula \eqref{e:first_variation_full} and a computation of the derivative of $W(u_{r,x_0})$ in $r$. We sketch the proof and we refer to \cite{rtv} for the detailed computations.
\begin{lemma}[Homogeneity of the blow-up limits]\label{l:weiss}
Let $u$ be a Lipschitz continuous solution of ~\eqref{e:fb_infty} in the open set $D\subset\R^d$ and let $x_0\in\partial\Omega_u\cap D$. Then, there is a constant $C>0$ such that
\begin{equation}\label{eq:weissmonot}
\frac{\partial}{\partial r}W(u_{x_0,r})\geq \frac{2}{r}\left[a^+\int_{\partial B_1}|x\cdot \nabla u^+_{x_0,r}-u^+_{x_0,r}|^2\,d\mathcal H^{d-1}+a^-\int_{\partial B_1}|x\cdot \nabla u^-_{x_0,r}-u^-_{x_0,r}|^2\,d\mathcal H^{d-1}\right]-C,
\end{equation}
where $a_+$ and $a_-$ are the nonnegative constants from \cref{l:limit_optimality} and 
$$u_{x_0,r}(x):=\frac1ru(x_0+rx).$$
As a consequence, if $u_0:\R^d\to\R$ is a blow-up limit of $u$ at $x_0$, then
\begin{enumerate}[\rm (i)]
\item if $a_+>0$, then $u_{0}^+$ is $1$-homogeneous;
\item if $a_->0$, then $u_{0}^-$ is $1$-homogeneous.
\end{enumerate}
\end{lemma}
\begin{proof}
The estimate \eqref{eq:weissmonot} follows directly from the first variation formula \eqref{e:first_variation_full}, just as in \cite[Lemma 5.37]{rtv}. Now, this implies that the function $r\mapsto W(u_{r,x_0})+Cr$ is non-decreasing and so the limit 
$$\Theta=\lim_{r\to0}W(u_{x_0,r})$$
exists. If $u_0$ is a blow-up limit of $u$ at $x_0$, $u_0$ is the locally uniform limit 
$$u_0=\lim_{n\to\infty}u_{r_n,x_0}\quad\text{for some sequence}\quad r_n\to0,$$
then, setting $(u_0)_\rho(x):=\ds\frac1\rho u(\rho x)$, we have 
$$W\big((u_0)_\rho\big)=\lim_{n\to\infty}W(u_{x_0,\rho r_n})=\Theta\quad\text{for every}\quad\rho>0.$$
On the other hand, we know that $u_0$ satisfies the optimality condition \eqref{e:first_variation_u0}. Thus, using again the computations from \cite[Lemma 5.37]{rtv}, we get that 
\begin{equation}\label{eq:weissmonot2}
\frac{\partial}{\partial r}W\big((u_0)_r\big)\geq \frac{2}{r}\left[a^+\int_{\partial B_1}|x\cdot \nabla (u_0)_r^+-(u_0)_r^+|^2\,d\mathcal H^{d-1}+a^-\int_{\partial B_1}|x\cdot \nabla (u_0)_r^--(u_0)_r^-|^2\,d\mathcal H^{d-1}\right].
\end{equation}
On the other hand, we know that $W\big((u_0)_r\big)$ is constant: $W\big((u_0)_r\big)=\Theta$ for every $r>0$. Thus the right-hand side of \eqref{eq:weissmonot2} is zero. This gives the claims (i) and (ii). 
\end{proof}

\subsection{Proof of \cref{t:main_intext}} We are now in position to prove \cref{t:main_intext}, which will imply \cref{c:main_cor} and conclude the proof of \cref{thm:main2} (and also the one of \cref{thm:main}). We proceed in several steps. \medskip

\noindent{\bf Step 1.\,\it The nondegeneracy of the coefficients \eqref{e:non-degeneracy-a} implies the classification of the blow-up limits \eqref{e:classification_blowup_limits} and \eqref{e:classification_blowup_limits2}.}\\
 Indeed, if $a_+>0$ and $a_->0$, then by \cref{l:weiss} any blow-up limit $u_0$ of $u$ at a two-phase point $x_0$ is one-homogeneous. Moreover, since $u_0$ is harmonic on $\Omega_0^+:=\{u_0>0\}$ and $\Omega_0^-:=\{u_0<0\}$, we have that it can be written in polar coordinates as 
$$u_0(r,\theta)=r\phi(\theta),$$
where the positive and the negative parts of $\phi:\mathbb S^{d-1}\to\R$ are non-zero (due to the nondegeneracy of $u_0^+$ and $u_0^-$) and are eigenfunctions on their supports, that is 
$$-\Delta_{\mathbb S^{d-1}}\phi_\pm=(d-1)\phi_\pm\quad\text{on}\quad\partial B_1\cap \Omega_0^\pm.$$
We now choose $\alpha$ and $\beta$ such that 
$$\int_{\mathbb S^{d-1}}\big(\alpha^2\phi_+^2+\beta^2\phi_-^2\big)\,d\theta=1\qquad\text{and}\qquad \int_{\mathbb S^{d-1}}\big(\alpha\phi_++\beta\phi_-\big)\,d\theta=0.$$
Moreover, integrating by parts on the sphere, we have that 
$$\int_{\mathbb S^{d-1}}\big|\nabla \big(\alpha\phi_++\beta\phi_-\big)\big|^2\,d\theta=d-1.$$
Now, by the variational formula for the eigenfunctions of the spherical Laplacian
$$d-1=\min\Big\{\int_{\mathbb S^{d-1}}|\nabla\psi|^2\,d\theta\ :\ \psi\in H^1(\mathbb S^{d-1}),\ \int_{\mathbb S^{d-1}}\psi\,d\theta=0,\  \int_{\mathbb S^{d-1}}\psi^2\,d\theta=1\Big\},$$
we get that the function 
$$\alpha\phi_++\beta\phi_-:\ \mathbb S^{d-1}\to\R,$$
is an eigenfunction of the Laplace-Beltrami operator corresponding to the eigenvalue $d-1$. Thus, $u_0^+$ and $u_0^-$ are linear functions, which gives \eqref{e:classification_blowup_limits}, that is, there are a unit vector $\nu\in\partial B_1$ and constants $\beta_+>0$ and $\beta_->0$ (notice that these constants are not a priori related to the auxiliary constants $\alpha$ and $\beta$ above) such that 
$$u_0(x):=\beta_+(x\cdot\nu)_+-\beta_-(x\cdot\nu)_-.$$
Now, in order to prove that $\beta_+$ and $\beta_-$ satisfy \eqref{e:classification_blowup_limits2}, we use the stationarity of $u_0$. Indeed, integrating by parts \eqref{e:first_variation_u0} we get that for every smooth vector field $\xi\in C^\infty_c(\R^d;\R^d)$, we have 
$$\int_{H_\nu}\Big(a_+|\nabla u_0^+|^2-a_-|\nabla u_0^-|^2\Big)\,\xi\cdot \nu\,d\HH^{d-1}=0,$$
where $H_\nu$ is the hyperplane $\ds \big\{x\in\R^d\ :\ x\cdot\nu=0\big\}$. Since the vector field $\xi$ is arbitrary, we get that
$$a_+\beta_+^2-a_-\beta_-^2=0.$$

\noindent{\bf Step 2.\,\it Strict positivity of the coefficients $a_+$ and $a_-$.}\\
Since $a_+\ge 0$, $a_-\ge 0$ and $a_++a_-=1$, we only need to exclude the case when one of the coefficients is zero and the other one is $1$. We argue by contradiction and we suppose that $a^-=0$ and $a^+=1$. We consider two cases. 

\noindent{\bf Step 2 - Case 1. \it There are no two-phase points in $D$.} In this case, we have that $\Omega_u^-$ and $\Omega_u^+$ lie at a positive distance in $D$. Now, if $a_-=0$, we have that 
$$\int_{\Omega_u^-}\text{\rm div}\,\xi\,dx=0\quad\text{for every}\quad \xi\in C^\infty_c\big(D\setminus \overline\Omega_u^+;\R^d\big).$$
Choosing vector fields of the form $(x-x_0)\phi_{\eps,r}(x-x_0)$, where the family of functions $\phi_{\eps,r}\in C^\infty_c(B_r)$ is such that 
$$\phi=1\quad\text{in}\quad B_{(1-\eps)r},\qquad \phi_{\eps,r}(x)=\frac{r-|x|}{\eps r}\quad\text{in}\quad B_r\setminus B_{(1-\eps)r},$$
and passing to the limit as $\eps\to0$, we get that (for almost-every $r>0$)
$$|B_r(x_0)\cap \Omega_u^-|=\frac{r}{d}\HH^{d-1}\big(\partial B_r(x_0)\cap \Omega_u^-\big).$$
Thus, the function
$$r\mapsto \frac{|B_r(x_0)\cap \Omega_u^-|}{|B_r|}$$
is constant for every $x_0\in D\setminus \overline\Omega_u^+$, which is impossible in the neighborhood of any one-phase point $x_0\in \partial\Omega_u^-\cap D$.\medskip

\noindent{\bf Step 2 - Case 2. \it There is at least one two-phase point $x_0\in\partial\Omega_u^+\cap\partial\Omega_u^-\cap D$.} Let $r>0$ be small enough such that $B_{r}(x_0)\subset D$ and let $y_0$ be any point such that 
$$y_0\in B_{\sfrac{r}2}(x_0)\qquad\text{and}\qquad y_0\in\Omega_u^-.$$
Let $z_0$ be the projection of $y_0$ at $\partial \Omega_u^+$. Notice that by construction, we have that $z_0\in D$. Let now $u_0$ be a blow-up limit of $u$ at $z_0$. Since $B_{\sfrac{r}2}(y_0)\cap \Omega_u^+=\emptyset$, we know that $u_0^+$ vanishes in the half space 
$$H_\nu^+=\{x\in\R^d\ :\ x\cdot\nu>0\},\qquad\text{where}\qquad \nu=\frac{z_0-y_0}{|z_0-y_0|}.$$
On the other hand, $u_0^+$ is harmonic in $\{u_0>0\}$ and, by \cref{l:weiss}, $1$-homogeneous. But then $u_0^+$ should be a linear function:
$$u_0^+(x)=c(x\cdot \nu)_+\quad\text{for every}\quad x\in\R^d,$$
for some positive constant $c$. Conversely, for the negative part $u_0^-$, we know that $\Omega_u^-$ lies in the oppposite half-space
$$H_\nu^-=\{x\in\R^d\ :\ x\cdot\nu<0\},$$
and that 
$$\int_{\Omega_u^-}\text{\rm div}\,\xi\,dx=0\quad\text{for every}\quad \xi\in C^\infty_c\big(H_\nu^-;\R^d\big).$$
Now, reasoning as in {Step\,2\,-\,Case 1} and knowing that $u_0^-$ is not identically zero in $B_1$ (due to the nondegeneracy of $u_-$), we get that $\Omega_u^-=H_\nu^-$. But now the optimality condition \eqref{e:first_variation_u0} gives that 
$$0=\int_{\partial H_\nu^+}a_+|\nabla u_0^+|^2(\xi\cdot\nu) \,d\HH^{d-1}=c^2\int_{\partial H_\nu^+}\xi\cdot\nu \,d\HH^{d-1},$$
for every smooth vector field $\xi\in C^\infty_c(\R^d;\R^d)$, which is a contradiction. This concludes the proof of Step 2. \medskip

\noindent{\bf Step\,3.\,\it Local inwards minimality of $u_+$.} We suppose that at least one of the one-phase free boundaries $\partial\Omega_u^+\setminus \partial\Omega_u^-$ and $\partial\Omega_u^-\setminus \partial\Omega_u^+$ is non-empty in $D$. Without loss of generality, there exists some point 
$$y_0\in D\cap \partial\Omega_u^+\setminus \partial\Omega_u^-.$$
Then, there is some $r>0$ such that $B_r(y_0)\cap\Omega_u^-=\emptyset$ and we can assume that $\partial\Omega_u^+$ is smooth in $B_r(y_0)$. Let now $\xi$ be a smooth vector field in $B_r(y_0)$ and let 
$$u_t(x)=u(\Psi_t(x))\qquad\text{where}\qquad \Phi_t=(Id+t\xi)^{-1}.$$. 
We can choose $\xi$ in such a way that  
$$|\{u_t>0\}\cap B_r(y_0)|-|\{u>0\}\cap B_r(y_0)|=t+o(t),$$
and 
$$\int_{B_r(y_0)}|\nabla u_t|^2\,dx-\int_{B_r(y_0)}|\nabla u|^2\,dx=-\frac1{a_+}t+o(t),$$
Now, suppose that $\rho$ is small enough and that $v\in H^1(B_\rho)$ is such that 
$$u=v\quad\text{on}\quad \partial B_\rho,\qquad v\le u\quad\text{in}\quad B_\rho,$$
and consider the test function 
$$\tilde v=v\quad\text{in}\quad  B_\rho,\qquad \tilde v=u\quad\text{in}\quad  D\setminus\big(B_\rho\cup B_r(y_0)\big),\qquad \tilde v=u_t\quad\text{in}\quad  B_r(y_0),$$
where $t$ is such that 
$$\big|\{u_t>0\}\cap B_r(y_0)\big|+\big|\{v>0\}\cap B_\rho|=\big|\{u>0\}\cap B_\rho\big|+\big|\{u>0\}\cap B_r(y_0)\big|,$$
and in particular, $t=O(\rho^d)$.
Thus, the minimality of $u$ implies that 
$$\frac{\ds\lambda_1(\Omega_u^+)-\int_{B_\rho}\big(|\nabla u|^2-|\nabla v|^2\big)-\int_{B_r(y_0)}\big(|\nabla u|^2-|\nabla u_t|^2\big)}{\ds1-\int_{B_\rho}\big(u^2-v^2\big)-\int_{B_r(y_0)}\big(u^2-u_t^2\big)}\ge\lambda_1(\Omega_u^+),$$
and so 
\begin{align*}
\int_{B_\rho}|\nabla v|^2\,dx&\ge \int_{B_\rho}|\nabla u|^2\,dx+ \int_{B_r(y_0)}\big(|\nabla u|^2-|\nabla u_t|^2\big)\,dx+o(\rho^d)\\
&\ge \int_{B_\rho}|\nabla u|^2\,dx+ \frac1{a_+}\Big(\big|\{u>0\}\cap B_\rho|-\big|\{v>0\}\cap B_\rho|\Big)+o(t)+o(\rho^d)\\
&= \int_{B_\rho}|\nabla u|^2\,dx+ \frac1{a_+}\Big(\big|\{u>0\}\cap B_\rho|-\big|\{v>0\}\cap B_\rho|\Big)+o(\rho^d).
\end{align*}
Thus, the rescaling $u_\rho(x)=\frac1\rho u(\rho x)$ satisfies 
\begin{align}\label{e:minimality_urho}
\int_{B_1}|\nabla v|^2\,dx\ge  \int_{B_1}|\nabla u_\rho|^2\,dx+ \frac1{a_+}\Big(\big|\{u_\rho>0\}\cap B_1\big|-\big|\{v>0\}\cap B_1\big|\Big)+o(1),
\end{align}
for all test functions $v$ such that 
$$v=u_\rho^+\quad\text{on}\quad \partial B_1\,,\quad v\le u_\rho^+\quad\text{in}\quad B_1\,.$$
\medskip

\noindent{\bf Step\,4.\,\it  Local inwards minimality of $u_0^+$.} Let $v:B_1\to\R$ be such that 
$$v=u_0^+\quad\text{on}\quad \partial B_1\,,\quad v\le u_0^+\quad\text{in}\quad B_1\,.$$
Fix $\eps>0$ and consider the function 
$$v_\eps: B_1\to\R,\quad v_\eps(x)=v(x)+\eps\big(|x|-(1-\eps)\big)_+$$
Thus, if $u_n:=u_{r_n}$ is a blow-up sequence that converges uniformly to $u_0^+$, then $v_\eps\ge u_n^+$ on $\partial B_1$, for every $n$. Thus, we can use $v_\eps\wedge u_n^+$ to test the minimality of $u_n^+$ in \eqref{e:minimality_urho}, thus obtaining that 
\begin{align}\label{e:minimality_un}
\int_{B_1}|\nabla (v_\eps\wedge u_n^+)|^2\,dx\ge  \int_{B_1}|\nabla u_n^+|^2\,dx+ \frac1{a_+}\Big(\big|\{u_n^+>0\}\cap B_1\big|-\big|\{v_\eps\wedge u_n^+>0\}\cap B_1\big|\Big)+o(1).
\end{align}
Now, using \cref{le:convH1bu} and passing to the limit as $n\to\infty$ gives
\begin{align*}
\int_{B_1}|\nabla (v_\eps\wedge u_0^+)|^2\,dx&\ge  \int_{B_1}|\nabla u_0^+|^2\,dx+ \frac1{a_+}\Big(\big|\{u_0>0\}\cap B_1\big|-\big|\{v_\eps\wedge u_0^+>0\}\cap B_1\big|\Big)\\
&\ge  \int_{B_1}|\nabla u_0^+|^2\,dx+ \frac1{a_+}\Big(\big|\{u_0>0\}\cap B_1\big|-\big|\{v>0\}\cap B_1\big|\Big)-\frac1{a_+}|B_{1}\setminus B_{1-\eps}|,
\end{align*}
which, since $\eps$ was arbitrary gives that 
$$\int_{B_1}|\nabla v|^2\,dx+ \frac1{a_+}\big|\{v>0\}\cap B_1\big|\ge \int_{B_1}|\nabla u_0^+|^2\,dx+ \frac1{a_+}\big|\{u_0^+>0\}\cap B_1\big|,$$
and concludes the proof of Step 4.\medskip

\noindent{\bf Step\,5.\,\it  Non-degeneracy of $\beta_+$ and $\beta_-$.} Let now $x_0$ be a two-phase point in $D$ and $u_0$ be a blow-up limit of $u$ at $x_0$. We know that $u_0$ is of the form \eqref{e:classification_blowup_limits}, where $\beta_+$ and $\beta_-$ are such that $a_+\beta_+^2=a_-\beta_-^2$. 
Let $\xi$ be any smooth vector field entering the half-space $H_\nu^+$. Then, the inwards minimizing property of $u_0^+$ implies that 
$$\int_{\partial H_\nu^+}\big(a_+\beta_+^2-1\big)|\xi\cdot \nu|\,d\HH^{d-1}\ge 0.$$
Thus, $a_+\beta_+^2\ge 1$ and, as a consequence, $a_-\beta_-^2\ge 1.$ This concludes the proof of  \cref{t:main_intext}.\qed

\end{document}